\pdfoutput=1
\documentclass[12pt]{article}

\usepackage{amsmath,amssymb,amsthm}
\usepackage{tikz-cd}
\usepackage{url}
\usepackage{verbatim}
\usepackage[export]{adjustbox}
\usepackage{array,booktabs}

\newdimen\R
\usepackage{lipsum}

\newcommand{\epislon}{\epsilon}

\DeclareMathOperator{\Hom}{Hom}
\DeclareMathOperator{\Ext}{Ext}
\DeclareMathOperator{\coker}{coker}
\newtheorem{definition}{Definition}
\newtheorem*{definition*}{Definition}
\newtheorem{theorem}{Theorem}
\newtheorem*{theorem*}{Theorem}
\newtheorem{corollary}{Corollary}
\newtheorem*{corollary*}{Corollary}
\newtheorem{lemma}{Lemma}
\newtheorem*{lemma*}{Lemma}
\newtheorem{proposition}{Proposition}
\newtheorem*{proposition*}{Proposition}

\newtheorem*{conjecture*}{Conjecture}
\theoremstyle{remark}
\newtheorem{example}{Example}
\newtheorem*{example*}{Example}
\newtheorem{remark}{Remark}
\newtheorem*{remark*}{Remark}

\newcommand{\canakci}{{\c{C}}anak{\c{c}}i }

\let\amsamp=&

\allowdisplaybreaks

\usepackage[a4paper,footskip=0.5in, margin = 0.7in]{geometry}

\title{Snake Graphs and Caldero-Chapoton Functions from Triangulated Orbifolds}

\usepackage{authblk}
\author[1]{Esther Banaian}
\author[2]{Yadira Valdivieso}
\affil[1]{\small{Department of Mathematics, Ny Munkegade 118, building 1530, 413, 8000 Aarhus C, Denmark}}
\affil[2]{\small{Deparment of  Actuarial Science, Physics, and Mathematics, Universidad de las Américas Puebla.
Sta. Catarina Mártir. San Andrés Cholula, Puebla. C.P. 72810. México}}

\begin{document}

\maketitle

\abstract{Generalized cluster algebras from orbifolds were defined by Chekhov and Shapiro to give a combinatorial description of their Teichmüller spaces. One can also assign a gentle algebra to a triangulated orbifold, as in the work of Labardini-Fragoso and Mou. In this work, we show that the Caldero-Chapoton map and the snake graph expansion map agree for arcs in triangulated orbifolds and arc modules, and similarly for closed curves and certain band modules. As a consequence, we have a bijection between some indecomposable modules over a gentle algebra and cluster variables in a generalized cluster algebra where both algebras arise from the same triangulated orbifold.}

\section{Introduction}

 Cluster algebras, originally defined by Fomin and Zelevinsky, are commutative rings with generators defined by a recursive process called mutation. 
Cluster algebras have been shown to have connections to a wide range of mathematical topics; in particular, they have a rich connection to the representation theory of finite dimensional algebras. For example, Caldero and Chapoton showed that non-initial cluster variables are in bijection with indecomposable modules of path algebras from quivers of ADE type \cite{CCMap}. They do this by describing a map (the \emph{Caldero-Chapoton map}) which sends a module to a Laurent polynomial. The definition of this map was extended in \cite{cerulli2015caldero} for modules over more general algebras. 

One well-studied class of cluster algebras are those arising from surfaces, where the mutation process is manifested as flipping arcs in a triangulation. The surface model was described by Fomin, Shapiro, and D. Thurston in \cite{FST-triangulatedSurfaces}. Given a triangulation $T$ of a surface with marked points $(S,M)$, Labardini-Fragoso associates a quiver $Q_T$ with potential $W_T$ whose mutation agrees with the flips of the arcs on the surface \cite{labardini2009quivers}. When all marked points of the surface are on the boundary $\partial S$ (that is, without punctures), the corresponding Jacobian algebra is gentle; these gentle algebras from triangulated surfaces were studied in  \cite{assem2010gentle}. It is natural to ask how the cluster algebra and gentle algebra from the same triangulated surface compare. This question was pursued in  \cite{geiss2020generic} and \cite{geiss2022schemes}. A related work is found in \cite{brustle2011cluster} where the authors compare the combinatorics of arcs on surfaces with features of the \emph{cluster category} associated to the cluster algebra from this surface. See \cite{buan2006tilting} and \cite{amiot2009cluster} for definitions of the cluster category.

Here, we study parallel stories where surfaces are replaced with \emph{orbifolds}. Chekhov and Shapiro provide a generalization of cluster algebras where the binomial exchange polynomial is replaced with a polynomial with arbitrarily many terms \cite{chekhov2014teichmuller}. A class of these algebras can be viewed as arising from triangulated orbifolds, in the same way that some cluster algebras arise from triangulated surfaces. 

 Labardini-Fragoso and Mou define a family of gentle algebras associated to triangulated unpunctured orbifolds with all orbifold points of order three \cite{labardini2023gentleI}. 
The algebras associated to a polygon with one orbifold point were previously studied in \cite{labardini2019family}. In this latter article, it was shown that functions given by applying the Caldero-Chapoton map to the modules of this algebra span the generalized cluster algebra from this orbifold. Our goal is to further investigate the connections between the generalized cluster algebra from a triangulated orbifold and the associated gentle algebra, motivated by the existing work in the surface case.

Our approach for this study is to compare the Caldero-Chapoton map with the snake graph expansion map. Snake graphs were introduced in \cite{Musiker-Schiffler} to give a combinatorial proof of positivity for cluster algebras from surfaces. They were extended by the first author and Kelley in \cite{banaian2020snake} for generalized cluster algebras from orbifolds. Our main result (Theorem \ref{thm:CCMapAgreesSnakeGraph}) is that the outputs of the Caldero-Chapoton map and the snake graph expansion map agree when applied to arc (string) modules and arcs respectively, and for certain band modules and closed curves.


The paper is organized as follows. In Section \ref{sec:GenCAAndOrbifold}, we review the basics of generalized cluster algebras and how we can build some of these from triangulated orbifolds. Section \ref{sec:SnakeGraph} builds on this by describing the snake graph expansion map, which explicitly gives the element of a generalized cluster algebra associated to any arc or closed curve on the orbifold. In Section \ref{sec:Strings}, we review definitions associated to gentle algebras and then discuss those which are associated to triangulated orbifolds. 

 Motivated by the comparison of string combinatorics and surface combinatorics in \cite{brustle2011cluster}, we show in Section  \ref{sec:ModuleCategory} that some similar results still hold for the  module category associated to an orbifold with marked points. This includes the fact that indecomposable modules correspond to arcs and closed curves (Theorem \ref{thm:StringsBijectCurves} and Corollary \ref{cor:ArcsGiveIndecomp}) and that Auslander-Reiten translation is encoded in rotating an arc (Theorem \ref{thm:ARTranslation}).

Section \ref{sec:CCMap} is devoted to proving Theorem \ref{thm:CCMapAgreesSnakeGraph}. To this end, in Theorem \ref{thm:BandMatchingSubmoduleBijection} we provide a bijection between the lattice of good matchings of a band graph and the lattice of canonical submodules of certain band modules. A similar such bijection between lattices of perfect matchings of snake graphs and canonical submodules of string modules is given in \cite{ccanakcci2021lattice}.

\section*{Acknowledgement}

The authors are grateful to Anna Felikson and Daniel Labardini-Fragoso for valuable conversations that improved this work.

The authors were supported by the European Union’s Horizon 2020
research and innovation programme under the Marie Sklodowska-Curie grant agreement No H2020-MSCA-IF-2018-838316. EB was also supported by the Independent Research Fund Denmark (1026-00050B).

YV would like to thank the Isaac Newton Institute for Mathematical Sciences, Cambridge, for support and hospitality during the programme “Cluster Algebras and
Representation Theory” where part of the work on this paper was undertaken. YV would also would like to thank the Homological Algebra group at Aarhus University and EB for its hospitality during her visit. EB would like to thank YV for her hospitality during her visit to University of Leicester. 

\section{Generalized Cluster Algebras and Orbifolds}\label{sec:GenCAAndOrbifold}

Here, we give a summary of the description of a generalized cluster algebra. We direct an interested reader to \cite{chekhov2014teichmuller} for the original definition. See also \cite{Fomin-Zelevinsky-I, williams2014cluster} for the definition of an ordinary cluster algebra.

A generalized cluster algebra of rank $n$ will be a subring of the ring of rational functions in $n$ variables. We let $\mathbb{P}$ denote a \emph{tropical semifield} and set $\mathcal{F} = \mathbb{Q}\mathbb{P}[x_1,\ldots,x_n]$. We define a \emph{generalized seed of rank $n$ with principal coefficients} (or just a \emph{seed}) to be a tuple $(\mathbf{x},\mathbf{y},B,\mathbf{d},Z)$ where 
\begin{itemize}
    \item $\textbf{x} = (x_1, \dots, x_n)$ is an $n$-tuple of elements of $\mathcal{F}$ which forms a free generating set for $\mathcal{F}$, 
    \item $\textbf{y} = (y_1, \dots, y_n)$ is an $n$-tuple with elements in $\mathbb{P}$, 
    \item $B = (b_{ij})$ is a skew-symmetrizable $n\times n$ integer matrix,
    \item $\mathbf{d} = (d_1,\ldots,d_n)$ is an $n$-tuple of positive integers, and 
    \item $\mathbf{Z} =(Z_1,\ldots,Z_n)$ is an $n$-tuple of polynomials with coefficients in $\mathbb{P}$ such that the degree of $Z_i$ is $d_i$.
\end{itemize}

We refer to $\textbf{x}$ as a \emph{cluster} with elements $x_1,\ldots,x_n$ being \emph{cluster variables}, $\textbf{y}$ as a \emph{coefficient tuple} with elements $y_1, \dots, y_n$ being \emph{coefficient variables}, $B$ as the \emph{exchange matrix}, $d_i \in \mathbf{d}$ as a \emph{mutation degree}, and $Z_i \in \mathbf{Z}$ as a \emph{exchange polynomials}. It is common to restrict to the case where the $Z_i$ are reciprocal and the leading coefficient is 1; this will be true in all the cluster algebras we consider.
If all mutation degrees are 1, then each $Z_i = 1$ and this will be a seed for an ordinary cluster algebra.

Given one seed, we can \emph{mutate} in direction $k$ for any $k \in [n]$ to produce a new seed. Informally, when we mutate a cluster $\mathbf{x} = (x_1,\ldots,x_n)$ in direction $k$ we replace $x_k$ with the product of a monomial and $Z_k$ evaluated at $\hat{y}_k$, where the monomial and the definition of $\hat{y}_k$ depend on the matrix $B$. The coefficient tuple and matrix will mutate as well. The assumption that each $Z_i$ is reciprocal with leading coefficient 1 will mean that the $Z_i$ are invariant under mutation.  See \cite{chekhov2014teichmuller} or \cite{Nakanishi} for precise mutation formulas.

Beginning with an initial seed, our generalized cluster algebra will have as generators all cluster variables in clusters mutation equivalent to the initial cluster.

\begin{definition}\label{def:GenCA}
Fix a semifield $\mathbb{P}$. Let $(\mathbf{x},\mathbf{y},B,\mathbf{d},\mathbf{Z})$ be a generalized seed of rank $n$. Then $\mathcal{A}(B,\mathbf{Z})$ is the ring generated by all cluster variables which can be reached by mutations starting from $\mathbf{x}$.
\end{definition}

We remark that, from the complete definition of mutation in a generalized cluster algebra, the algebra produced from an initial seed $(\mathbf{x},\mathbf{y},B,\mathbf{d},\mathbf{Z})$ is determined  by the matrix $B$ and tuple of polynomials $\mathbf{Z}$. Thus, the notation $\mathcal{A}(B,\mathbf{Z})$ is unambiguous.

Generalized cluster algebras retain the Laurent Phenomenon, a celebrated property of ordinary cluster algebras and a subclass have been shown to exhibit positivity \cite{chekhov2014teichmuller}. This subclass includes generalized cluster algebras from orbifolds, as will be defined in Section \ref{sec:GenCAFromOrb}.

\subsection{Generalized Cluster Quivers}
In the case of an ordinary cluster algebra, if the exchange matrix is skew-symmetric we can instead define the cluster algebra in terms of a \emph{quiver}, $Q = (Q_0,Q_1)$ which is a directed graph with vertices $Q_0$ labeled by $1,\ldots,n$ and directed arrows $Q_1$. Let $s,t: Q_1 \to Q_0$ denote the start and terminal point of an arrow, so that $\alpha \in Q_1$ looks like $s(\alpha) \xrightarrow{\alpha} t(\alpha)$. A quiver is called a \emph{cluster quiver} if it has no loops (no $\alpha \in Q_1$ with $s(\alpha) = t(\alpha)$) and no directed 2-cycles (no distinct pair $\alpha,\beta \in Q_1$ with $s(\alpha) = t(\beta)$ and $s(\beta) = t(\alpha)$). 

For notational convenience in the later parts of the paper, we introduce a \emph{generalized cluster quiver}, which will encode a skew-symmetric exchange matrix for a generalized cluster algebra with all mutation degrees 1 or 2. This restriction is simply because these are the generalized cluster algebras we work with; the definition could be generalized to allow arbitrary mutation degrees.

\begin{definition}[Generalized Cluster Quiver]\label{def:GenClQuiver}
 A \emph{generalized cluster quiver} $Q = (Q_0,Q_1)$ is a cluster quiver with $Q_0$ partitioned into two sets $Q_0 = S \sqcup P$.  We denote an arbitrary vertex in $S$ with $i$ and call this a ``standard vertex''. We denote an arbitrary vertex in $P$ with $\bar{i}$ and call this  a ``pending vertex''. 
\end{definition}

The language of a standard vertex and a pending vertex will be justified in Section \ref{sec:GenCAFromOrb}. The difference between cluster quivers and generalized cluster quivers is made clear with the mutation rules.

\begin{definition}[Mutation]\label{def:GenCAMutation}
Let $Q$ be a generalized cluster quiver and let $k \in Q_0$. We define a new generalized cluster quiver $\mu_k(Q)$ through the following process.
\begin{enumerate}
    \item If $k \in S$ ($k \in P$), then for each path $i \to k \to j$ in $Q$, add one (two) arrow(s) $i \to j$. 
    \item Reverse the direction of all arrows incident to $k$.
    \item Delete any directed 2-cycles formed in steps (1) and (2).
\end{enumerate}
\end{definition}

Note that mutation of a generalized cluster quiver does not affect the partitioning of vertices $Q_0 = S \sqcup P$. If $P = \emptyset$, this is ordinary cluster quiver mutation. See Figure \ref{fig:GenClQuivMutation} for a comparison in mutating a standard and a pending vertex in a three cycle.

\begin{figure}
\begin{center}
\begin{tabular}{ccc}
\begin{tikzcd}[arrow style=tikz,>=stealth,row sep=4em]
2 \arrow[dr]
&& 3 \arrow[ll]  \\
& 1 \arrow[ur]
\end{tikzcd} & $\xrightarrow{\mu_1}$ &
\begin{tikzcd}[arrow style=tikz,>=stealth,row sep=4em]
2 && 3 \arrow[dl]  \\
& 1 \arrow[ul]
\end{tikzcd}\\
\begin{tikzcd}[arrow style=tikz,>=stealth,row sep=4em]
2 \arrow[dr]
&& 3 \arrow[ll]  \\
& \bar{1} \arrow[ur]
\end{tikzcd} & $\xrightarrow{\mu_1}$ &
\begin{tikzcd}[arrow style=tikz,>=stealth,row sep=4em]
2\arrow[rr] && 3 \arrow[dl]  \\
& 1 \arrow[ul]
\end{tikzcd}\\
\end{tabular}
\end{center}
\caption{Examples of mutating at a standard vertex and a pending vertex.}\label{fig:GenClQuivMutation}
\end{figure}
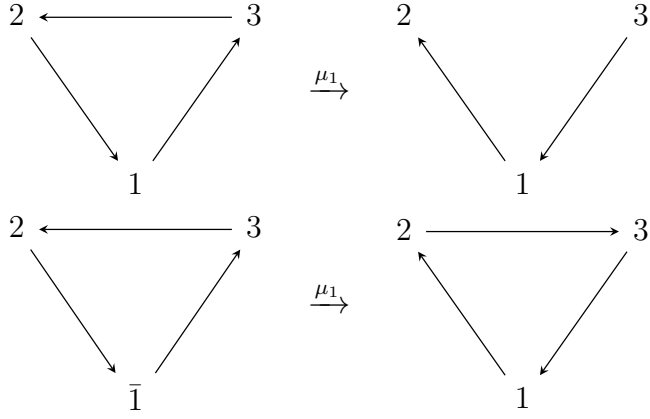

Given an $n\times n$ skew-symmetric matrix $B$ and a tuple $\mathbf{d} = (d_1,\ldots,d_n)$ with $d_i \in \{1,2\}$, we define a generalized cluster quiver $Q_{B,\mathbf{d}}$. We set $(Q_{B,\mathbf{d}})_0 = [n]$, and for $i \in [n]$, $i \in S$ if and only if $d_i = 1$. For any entry $b_{i,j} > 0$, we include $b_{i,j}$ arrows $i \to j$ in $Q_{B,\mathbf{d}}$. A quick check shows that mutation of matrices in such a generalized cluster algebra is compatible with the mutation of generalized cluster quivers. The definitions of matrix mutation for a generalized cluster algebra are given in \cite{chekhov2014teichmuller}.

\begin{lemma}\label{lem:GenClusterQuiverEqualsMatrix}
Let $B$ be an $n\times n$ skew-symmetric matrix and let $\mathbf{d} = (d_1,\ldots,d_n)$ be a tuple such that $d_i \in \{1,2\}$ for each $i$. 
Then, for all $k \in [n]$, $\mu_k(Q_{B,\mathbf{d}}) = Q_{\mu_k(B),\mathbf{d}}$.
\end{lemma}

\subsection{Orbifolds}

An orbifold is a generalization of a manifold where the local structure is given by quotients of open subsets of $\mathbb{R}^n$ under finite group actions. For our considerations, an orbifold is a marked surface with an additional set of special points called \emph{orbifold points} which denote fixed points of the group action.

\begin{definition}\label{def:Orbifold}
An \emph{orbifold} $\mathcal{O}$ is a triple $(S,M,O)$,  where $S$ is a bordered surface, $M$ is a finite set of marked points, and $O$ is a finite set of orbifold points, such that
\begin{enumerate}
    \item no point is both a marked point and an orbifold point (i.e., $M \cap O) = \emptyset$),
    \item all orbifold points are interior points of $S$, and
    \item each boundary component of $S$ contains at least one marked point.
\end{enumerate} 
Moreover, each orbifold point comes with an associated order $p \in \mathbb{Z}_{\geq 2}$. 
\end{definition}

Here, we consider that all marked points lie on the boundary; that is, there are no \emph{punctures}.

 \begin{definition}\label{def:Arc}
An \emph{arc} $\gamma$ on an orbifold $\mathcal{O} = (S,M,O)$  is a non-self-intersecting curve in $S$ with endpoints in $M$ that is otherwise disjoint from $M$, $O$, and $\partial \mathcal{O}$. Curves that are contractible onto $\partial \mathcal{O}$ are not considered arcs. Arcs are considered up to isotopy class. An arc which cuts out an unpunctured monogon with exactly one point in $O$ is called a \emph{pending arc}, while all other arcs are called \emph{standard arcs}.
\end{definition}

In some sources, pending arcs are instead drawn as a line between a marked point and an orbifold point;  such curves are not allowed in our definition of an arc.

\begin{definition}\label{def:Triangulation}
Two arcs are considered \emph{compatible} if their isotopy classes contain non-intersecting representatives. A \emph{triangulation} is a maximal collection of pairwise compatible arcs.
\end{definition}

A triangulation of an orbifold $\mathcal{O}$ without punctures contains three types of triangles, shown in Figure \ref{fig:OrbifoldTriangles}. Note that if an orbifold $\mathcal{O}$ has $\ell$ orbifold points, each triangulation of $\mathcal{O}$ will have exactly $\ell$ pending arcs.

\begin{figure}
    \centering
\begin{tabular}{c|c|c|c}
    $T$ & \begin{tikzpicture}[scale=1.25]
     \node[above] at (0.5,0.5){2};
     \draw(1,1) to (0,0) to (2,0) to (1,1);
     \node[below] at (1,0) {1};
     \node[above] at (1.5,0.5){3};
    \end{tikzpicture} & \begin{tikzpicture}[scale = 1.25]
\draw[out=30,in=-30,looseness=1.5] (0,0.3) to (0,2.5);
\draw[out=150,in=-150,looseness=1.5] (0,0.3) to (0,2.5);
\draw[in=0,out=45,looseness =1.25] (0,0.3) to (0,2);
\draw[in=180,out=135,looseness=1.25] (0,0.3) to (0,2);
\draw[thick] (0,1.7) node {$\mathbf{\times}$};
\node[right] at (0.8,1.5){$\tau_3$};
\node[left] at (-0.8,1.5){$\tau_2$};
\node[above] at (0,2){$\tau_1$};
\end{tikzpicture} &
\begin{tikzpicture}[scale=2]
\draw[out=75,in=-180] (0,0) to (0.3,0.85);
\draw[out=0,in=0, looseness = 1] (0,0) to node[midway,left,yshift=-4pt]{$2$} (0.3,0.85);
\draw[out=105,in=0] (0,0) to (-0.3,0.85);
\draw[out=180,in=180, looseness = 1] (0,0) to node[midway,right,yshift=-4pt]{$1$} (-0.3,0.85);
\draw[looseness=1.75,out=0,in=0] (0,0) to (0,1.3);
\draw[looseness=1.75,out=180,in=180] (0,1.3) to (0,0);
\draw[thick] (-0.3,0.65) node {$\mathbf{\times}$};
\draw[thick] (0.3,0.65) node {$\mathbf{\times}$};
\node [yshift=8pt] at (0,1) {$3$};
\end{tikzpicture}
\\ \hline
   $Q^T_{gen}$  & 
\begin{tikzcd}[arrow style=tikz,>=stealth,row sep=4em]
2 \arrow[dr,"\alpha"]
&& 3 \arrow[ll,"\beta"]  \\
& 1 \arrow[ur,"\gamma"]
\end{tikzcd} &  
\begin{tikzcd}[arrow style=tikz,>=stealth,row sep=4em]
2 \arrow[dr,"\alpha"]
&& 3 \arrow[ll,"\beta"]  \\
& \bar{1} \arrow[ur,"\gamma"]
\end{tikzcd}& 
\begin{tikzcd}[arrow style=tikz,>=stealth,row sep=4em]
\bar{2} \arrow[dr,"\alpha"]
&& 3 \arrow[ll,"\beta"]  \\
& \bar{1} \arrow[ur,"\gamma"]
\end{tikzcd}
\\
\end{tabular}
    \caption{Three types of triangles in a triangulation of an unpunctured orbifold and corresponding generalized cluster quiver pieces.}\label{fig:OrbifoldTriangles}
\end{figure}
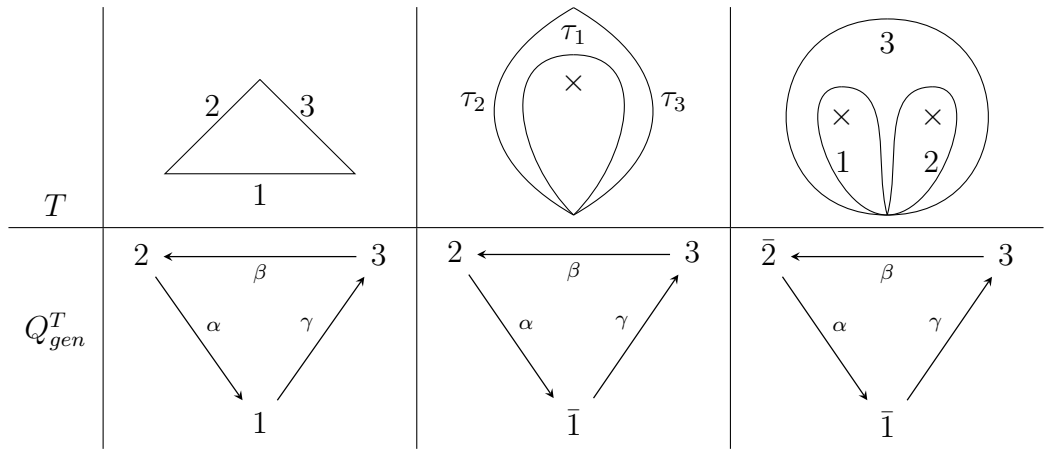

The order of an orbifold point will control in which ways an arc (possibly with self-intersection) can wind around the orbifold point. That is, an arc which winds $k$ times around an orbifold point of order $p$ is isotopic to an arc which winds $k-p$ times. In this paper, every orbifold point will have order three. Thus, by repeatedly using this equivalence, we will only see the following types of winding around around orbifold points.

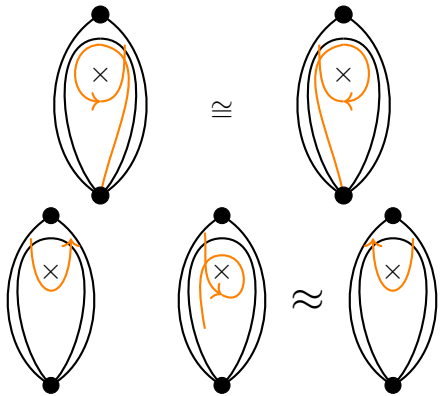
\begin{center}
\begin{tikzpicture}[scale = 1.6]

\draw[out=30,in=-30, thick] (0,0) to (0,1.5);
\draw[out=150,in=-150, thick] (0,0) to (0,1.5);


\draw[in=0,out=45,looseness =0.9, thick] (0,0) to (0,1.3);
\draw[in=180,out=135,looseness=0.9, thick] (0,0) to (0,1.3);

\draw[thick,orange, thick,out=80,in=-10,looseness = 1] (0,0) to (0,1.25);
\draw[thick,orange, thick,out=180,in=180,looseness = 1.5, ->] (0,1.25) to (0,0.78);
\draw[thick,orange, thick,out=0,in=270,looseness = 1,=] (0,0.78) to (0.2,1.25);

\draw[fill=black] (0,0) circle [radius=2pt];
\draw[fill=black] (0,1.5) circle [radius=2pt];
\draw[thick] (0,1) node {$\mathbf{\times}$};

\node[] at (1,0.7) {$\cong$};

\draw[thick,out=30,in=-30] (2,0) to (2,1.5);
\draw[thick,out=150,in=-150] (2,0) to (2,1.5);


\draw[thick,in=0,out=45,looseness =0.9] (2,0) to (2,1.3);
\draw[thick,in=180,out=135,looseness=0.9] (2,0) to (2,1.3);

\draw[thick,orange, thick,out=100,in=190,looseness = 1] (2,0) to (2,1.25);
\draw[thick,orange, thick,out=0,in=0,looseness = 1.5, ->] (2,1.25) to (2,0.78);
\draw[thick,orange, thick,out=180,in=270,looseness = 1] (2,0.78) to (1.8,1.25);

\draw[fill=black] (2,0) circle [radius=2pt];
\draw[fill=black] (2,1.5) circle [radius=2pt];
\draw[thick] (2,1) node {$\mathbf{\times}$};

\end{tikzpicture}

 \begin{tikzpicture}[scale = 1.5]

\tikzset{->-/.style={decoration={
  markings,
  mark=at position #1 with {\arrow{>}}},postaction={decorate}}}

\draw[out=30,in=-30, thick] (-1,0) to (-1,1.5);
\draw[out=150,in=-150, thick] (-1,0) to (-1,1.5);

\draw[in=0,out=45,looseness =0.9, thick] (-1,0) to (-1,1.3);
\draw[in=180,out=135,looseness=0.9, thick] (-1,0) to (-1,1.3);

\draw[orange, thick,out=-90,in=-90,looseness = 4.5,->=0.25] (-1.175,1.3) to (-0.825,1.3);

\draw[fill=black] (-1,0) circle [radius=2pt];
\draw[fill=black] (-1,1.5) circle [radius=2pt];
\draw[thick] (-1,1) node {$\mathbf{\times}$};


\draw[out=30,in=-30, thick] (0.5,0) to (0.5,1.5);
\draw[out=150,in=-150, thick] (0.5,0) to (0.5,1.5);

\draw[in=0,out=45,looseness =0.9, thick] (0.5,0) to (0.5,1.3);
\draw[in=180,out=135,looseness=0.9, thick] (0.5,0) to (0.5,1.3);

\draw[orange, thick,out=100,in=180,looseness = 1] (0.35,0.5) to (0.5,1.15);
\draw[orange, thick,out=0,in=-30,looseness = 2] (0.5,1.15) to (0.5,0.8);
\draw[orange, thick,in=150,out=275,looseness = 1,->] (0.35,1.35) to (0.5,0.8) ;

\draw[fill=black] (0.5,0) circle [radius=2pt];
\draw[fill=black] (0.5,1.5) circle [radius=2pt];
\draw[thick] (0.5,1) node {$\mathbf{\times}$};


\draw[out=30,in=-30, thick] (2,0) to (2,1.5);
\draw[out=150,in=-150, thick] (2,0) to (2,1.5);

\draw[in=0,out=45,looseness =0.9, thick] (2,0) to (2,1.3);
\draw[in=180,out=135,looseness=0.9, thick] (2,0) to (2,1.3);

\draw[orange, thick,out=-90,in=-90,looseness = 4.5,->] (3.675-1.5,1.3) to (3.325-1.5,1.3);

\draw[fill=black] (2,0) circle [radius=2pt];
\draw[fill=black] (2,1.5) circle [radius=2pt];
\draw[thick] (2,1) node {$\mathbf{\times}$};

\node[scale=1.5] at (2.75-1.5,0.75) {$\approx$};

\end{tikzpicture}
\end{center}

Given a triangulation $T = \{\tau_1,\ldots,\tau_n\}$, we can \emph{flip} an arc $\tau_k$ by removing it and replacing it with a unique arc such that $\mu_k(T):= (T \backslash \{\tau_k\}) \cup \{\tau_k'\}$ is a new triangulation of $\mathcal{O}$. Felikson, Shapiro, and Tumarkin show that flips act transitively on the set of triangulations of $\mathcal{O}$ \cite{FST-orbTriang}. 
This implies that every triangulation of an orbifold $\mathcal{O}$ has the same number of arcs. Moreover, since  the flip of a standard (pending) arc will always be another standard (pending) arc, the numbers of standard and pending arcs in any triangulation of a given orbifold are also constant.

\subsection{Generalized Cluster Algebra from Orbifold}\label{sec:GenCAFromOrb}

We describe how to produce a generalized cluster algebra from a triangulated orbifold following \cite{chekhov2014teichmuller}. Let $\mathcal{O}$ be an orbifold and fix a triangulation $T = \{\tau_1,\ldots,\tau_n\}$. We form a generalized cluster quiver $Q^T_{\text{gen}} = (Q_0,Q_1)$ based on $T$. We identify $Q_0$ with $[n]$ so that each arc $\tau_i$ in $T$ has a corresponding vertex $i$. We partition $Q_0$ by setting $S \subseteq [n]$ to be the set of vertices $i$ such that $\tau_i$ is standard and $P \subseteq [n]$ to be the set of vertices $j$ such that $\tau_j$ is pending. If $\tau_i$ and $\tau_j$ are edges of the same triangle in the triangulation, and $\tau_j$ immediately follows $\tau_i$ in counterclockwise order, we set an edge $i \to j$. See Figure \ref{fig:OrbifoldTriangles}.

The operations of flipping arcs in an orbifold and mutating a generalized cluster quiver agree. As an illustration, compare Figures \ref{fig:GenClQuivMutation} and \ref{fig:MutateStdAndPendArcs}. 

\begin{lemma}\label{lem:FlipArcMutateQuiver}
Let $T = \{\tau_1,\ldots,\tau_n\}$ and let $Q^T_{\text{gen}}$ be the associated generalized cluster quiver. Let $T'$ be the effect of flipping arc $\tau_k$ in $T$ for any $1 \leq k \leq n$. Then, $Q^{T'}_{\text{gen}} = \mu_k(Q^T_{\text{gen}})$. 
\end{lemma}

Labardini-Fragoso and Mou provide a related result using skew-symmetrizable matrices in \cite{labardini2023gentleI}.

In order to define a generalized cluster algebra, we still  must define the exchange polynomials $Z_i$. In the case when all mutation degrees are 1 or 2, we only need to set pick terms $z_{i,1}$ where $Z_i(u) = 1 + z_{i,1} u + u^2$ whenever $d_i = 2$. We know that $d_i = 2$ whenever $\tau_i$ is a pending arc. In this case, if the orbifold point enclosed by $\tau_i$ has order $p$, we set $z_{i,1} = 2\cos(\pi/p)$. This choice of coefficient is shown by Chekhov and Shapiro to give the correct expression for how the $\lambda$-lengths of the orbifold change as we flip arcs. In our setting with all orbifold points having order 3, all coefficients $z_{i,1}$ will be $1$. Therefore, given a triangulation $T$ of an orbifold, we can denote the associated generalized cluster algebra simply as $\mathcal{A}(Q^T_{\text{gen}})$.

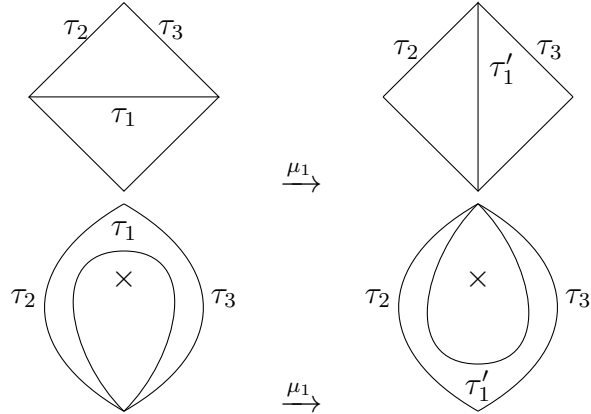
\begin{figure} 
    \centering
\begin{tabular}{ccc}
\begin{tikzpicture}[scale=1.25]
     \node[above] at (0.5,0.5){$\tau_2$};
     \draw(1,1) to (0,0) to (2,0) to (1,1);
     \draw (0,0) to (1,-1) to (2,0);
     \node[below] at (1,0) {$\tau_1$};
     \node[above] at (1.5,0.5){$\tau_3$};
    \end{tikzpicture} & $\xrightarrow{\mu_1}$ &
\begin{tikzpicture}[scale=1.25]
     \node[above, left] at (0.5,0.5){$\tau_2$};
     \draw (1,-1) to (1,1) to (0,0);
     \draw (0,0) to (1,-1) to (2,0);
     \draw(1,1) to (2,0);
     \node[right] at (1,0.3) {$\tau_1'$};
     \node[above, right] at (1.5,0.5){$\tau_3$};
    \end{tikzpicture}\\
\begin{tikzpicture}[scale = 1.25]
\draw[out=30,in=-30,looseness=1.5] (0,0.3) to (0,2.5);
\draw[out=150,in=-150,looseness=1.5] (0,0.3) to (0,2.5);
\draw[in=0,out=45,looseness =1.25] (0,0.3) to (0,2);
\draw[in=180,out=135,looseness=1.25] (0,0.3) to (0,2);
\draw[thick] (0,1.7) node {$\mathbf{\times}$};
\node[right] at (0.8,1.5){$\tau_3$};
\node[left] at (-0.8,1.5){$\tau_2$};
\node[above] at (0,2){$\tau_1$};
\end{tikzpicture} & $\xrightarrow{\mu_1}$ &
\begin{tikzpicture}[scale = 1.25]
\draw[out=30,in=-30,looseness=1.5] (0,0.3) to (0,2.5);
\draw[out=150,in=-150,looseness=1.5] (0,0.3) to (0,2.5);
\draw[in=0,out=45-90,looseness =1.25] (0,2.5) to (0,0.8);
\draw[in=180,out=135+90,looseness=1.25] (0,2.5) to (0,0.8);
\draw[thick] (0,1.7) node {$\mathbf{\times}$};
\node[right] at (0.8,1.5){$\tau_3$};
\node[left] at (-0.8,1.5){$\tau_2$};
\node[below] at (0,0.85){$\tau_1'$};
\end{tikzpicture}\\
\end{tabular}  
    \caption{Comparing examples of mutating standard and pending arcs. Unlabeled edges are assumed to be boundary edges. Compare these mutations to the generalized quiver mutations in Figure \ref{fig:GenClQuivMutation}} \label{fig:MutateStdAndPendArcs}
\end{figure}

\section{Snake Graphs}\label{sec:SnakeGraph}

 A snake graph is a bipartite graph whose set of \emph{perfect matchings} encodes the expansion of a cluster variable in a cluster algebra of surface type. Snake graphs were defined by Musiker and Schiffler in \cite{Musiker-Schiffler} for unpunctured surfaces and later by Musiker, Schiffler, and Williams for punctured surfaces \cite{MSW}. Band graphs were defined in \cite{MSW-bases} to give expansions for elements of a cluster algebra given by closed curves on a surface. These constructions were extended to generalized cluster algebras arising from unpunctured orbifolds (with orbifold points order at least 3) in \cite{banaian2020snake}. We review this latter construction and also recall the partial order which can be placed on the set of perfect matchings of a snake or band graph. 
\subsection{Snake Graphs and Band Graphs from Orbifolds}\label{subsec:SnakeGraphsfromOrbifold}

Here we provide a simplified version of the construction from \cite{banaian2020snake} for snake graphs from orbifolds where all points are order three. 

Fix a triangulation $T = \{\tau_1,\ldots,\tau_n\}$ of an orbifold $\mathcal{O}$. Let $\{\tau_{n+1},\ldots,\tau_{n+c}\}$ be the boundary arcs of $\mathcal{O}$. Let $\gamma$ be an arc on $\mathcal{O}$, possibly with self-intersection. We choose a representative of $\gamma$ from its isotopy class which has minimal self-intersections and minimal intersections with $T$. 

If $\gamma = \tau_i$ for $1 \leq i \leq n$, then $\mathcal{G}_{\gamma,T}$ is a single edge labeled $\tau_i$. Otherwise, $\gamma \notin T$ and therefore $\gamma$ must cross at least one arc in $T$.  

Choose an orientation for $\gamma$. Let $p_1,\ldots,p_d$ be the list of points in $\gamma \cap T$, with order given by the orientation of $\gamma$. Let $\tau_{i_1},\ldots, \tau_{i_d}$ be such that $p_j$ lies on $\tau_{i_j}$; note the same arc from $T$ can appear multiple times in this list. For each standard arc $\tau_{i_j}$ that $\gamma$ crosses, we construct a square tile $G_{j}$ by taking the two triangles that $\tau_{i_j}$ borders and gluing them along $\tau_{i_j}$ such that either both have the same orientation as in $\mathcal{O}$ or neither. We say that the edge of the snake graph associated to $\tau_i$ is \emph{weighted} with variable $x_i$. We include a dashed line labeled with $i$; this is not considered an edge of the graph but rather a label for this tile. See the two tiles below. We say that the square tile $G_j$ has relative orientation $+1$ if the orientation of its triangles matches that of $\mathcal{O}$ and $-1$ otherwise. We denote this as $\textrm{rel}(G_j) = \pm 1.$

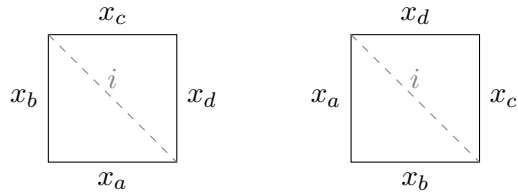
\begin{center}
\begin{tikzpicture}[scale = .14cm]
\draw (45:0.3cm) to node[above]{$x_c$}(135:0.3cm) to node[left]{$x_b$}(180+45:0.3cm) to node[below]{$x_a$}(270+45:0.3cm) to node[right]{$x_d$} (45:0.3cm);
\draw[gray, dashed](135:0.3cm) to node[above]{$i$} (315:0.3cm);
\draw[xshift = 1 cm] (45:0.3cm) to node[above]{$x_d$}(135:0.3cm) to node[left]{$x_a$}(180+45:0.3cm) to node[below]{$x_b$}(270+45:0.3cm) to node[right]{$x_c$} (45:0.3cm);
\draw[xshift = 1 cm, gray, dashed](135:0.3cm) to node[above]{$i$} (315:0.3cm);
\end{tikzpicture}
\end{center}

Next we consider the case when $\tau_{i_j}$ is a pending arc.   If $\tau_{i_1}$ is a pending arc and $\tau_{i_1} \neq \tau_{i_2}$, or if $\tau_{i_d}$ is a pending arc and $\tau_{i_{d-1}} \neq \tau_{i_d}$, we have a configuration as in Figure \ref{fig:WindBeforeFirst}. Since the orbifold point is order three, up to isotopy there is only one such local configuration for $\gamma$. 

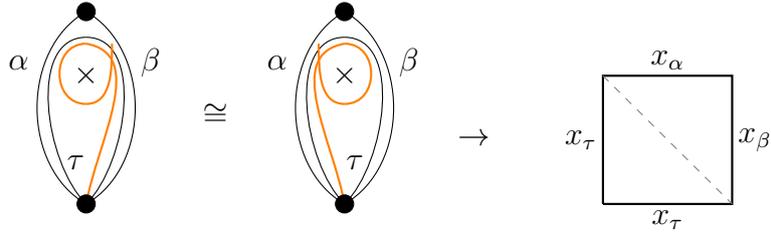
\begin{figure}
\centering
\begin{tikzpicture}[scale = 1.7]
\draw[out=30,in=-30] (0,0) to (0,1.5);
\draw[out=150,in=-150] (0,0) to (0,1.5);
\node[left,xshift=2pt] at (-0.4,1.1) {$\alpha$};
\node [right,xshift=-3pt] at (0.4,1.1) {$\beta$};
\node [below,yshift=3pt,xshift=1pt] at (-0.1,0.4) {$\tau$};
\draw[in=0,out=45,looseness =0.9] (0,0) to (0,1.3);
\draw[in=180,out=135,looseness=0.9] (0,0) to (0,1.3);
\draw[orange, thick,out=80,in=-10,looseness = 1] (0,0) to (0,1.25);
\draw[orange, thick,out=180,in=180,looseness = 1.5] (0,1.25) to (0,0.78);
\draw[orange, thick,out=0,in=270,looseness = 1] (0,0.78) to (0.2,1.25);
\draw[fill=black] (0,0) circle [radius=2pt];
\draw[fill=black] (0,1.5) circle [radius=2pt];
\draw[thick] (0,1) node {$\mathbf{\times}$};

\node[] at (1,0.7) {$\cong$};

\draw[out=30,in=-30] (2,0) to (2,1.5);
\draw[out=150,in=-150] (2,0) to (2,1.5);

\node[left,xshift=2pt] at (1.6,1.1) {$\alpha$};
\node [right,xshift=-3pt] at (2.4,1.1) {$\beta$};
\node [below,yshift=3pt,xshift=-1pt] at (2.1,0.4) {$\tau$};
\draw[in=0,out=45,looseness =0.9] (2,0) to (2,1.3);
\draw[in=180,out=135,looseness=0.9] (2,0) to (2,1.3);

\draw[orange, thick,out=100,in=190,looseness = 1] (2,0) to (2,1.25);
\draw[orange, thick,out=0,in=0,looseness = 1.5] (2,1.25) to (2,0.78);
\draw[orange, thick,out=180,in=270,looseness = 1] (2,0.78) to (1.8,1.25);

\draw[fill=black] (2,0) circle [radius=2pt];
\draw[fill=black] (2,1.5) circle [radius=2pt];
\draw[thick] (2,1) node {$\mathbf{\times}$};

\node[] at (3,0.5) {$\to$};

\draw[thick] (4,0) to node[below,midway,yshift = 2pt]{$x_\tau$} (5,0) to node[right, midway, xshift = -2pt]{$x_\beta$} (5,1) to node[above, midway, yshift = -2pt]{$x_\alpha$} (4,1) to node[left, midway, xshift = 2pt]{$x_\tau$} (4,0);
\draw[gray,dashed] (4,1) to (5,0);
\end{tikzpicture}
\caption{The tile for a snake graph when an arc winds around an orbifold point before its first crossing with $T$ or after its last crossing.  The tile shown is in the positive orientation; one can flip the edges labeled $x_\alpha$ and $x_\beta$ for the tile with negative orientation. }\label{fig:WindBeforeFirst}
\end{figure}

The other case is when $\gamma$ crosses a pending arc $\tau_{i_j}$ twice consecutively; that is, $\tau_{i_j} = \tau_{i_{j-1}}$ or $\tau_{i_j} = \tau_{i_{j+1}}$. Since we are considering only orbifold points of order three, up to isotopy we have only four possible configurations of crossings as given in Table \ref{table:SnakeGraphPuzzlePiece}. In the case when this pair of crossings is the first or last two crossings of $\gamma$ and $T$, one can also obtain the relevant piece of the snake graph by taking an appropriate entry from the table and removing one of the tiles labeled $\alpha$ or $\beta$.

\begin{center}
\begin{table}
{\renewcommand{\arraystretch}{2}
\centering
\begin{tabular}{|c|c|c|c|}
\hline
 
\begin{tikzpicture}[scale = 1.7]

\tikzset{->-/.style={decoration={
  markings,
  mark=at position #1 with {\arrow{>}}},postaction={decorate}}}

\draw[out=30,in=-30,looseness=1.5] (0,0.3) to (0,1.5);
\draw[out=150,in=-150,looseness=1.5] (0,0.3) to (0,1.5);

\draw[in=0,out=45,looseness =1.25] (0,0.3) to (0,1.3);
\draw[in=180,out=135,looseness=1.25] (0,0.3) to (0,1.3);

\draw[orange,thick,out=-15,in=195, ->] (-0.5,0.7) to (0.5,0.7);

\draw[fill=black] (0,0.3) circle [radius=2pt];
\draw[fill=black] (0,1.5) circle [radius=2pt];
\draw[thick] (0,1) node {$\mathbf{\times}$};

\node[] at (-0.26,1.4) {$\alpha$};
\node[] at (0.26,1.41) {$\beta$};
\node[] at (-0.2,0.85) {$\tau$};

\end{tikzpicture}
 & 
 \begin{tikzpicture}[scale = 1.7]

\tikzset{->-/.style={decoration={
  markings,
  mark=at position #1 with {\arrow{>}}},postaction={decorate}}}


\draw[out=30,in=-30,looseness=1.5] (0,0.3) to (0,1.5);
\draw[out=150,in=-150,looseness=1.5] (0,0.3) to (0,1.5);

\draw[in=0,out=45,looseness =1.25] (0,0.3) to (0,1.3);
\draw[in=180,out=135,looseness=1.25] (0,0.3) to (0,1.3);

\draw[orange,thick,out=-15,in=270] (-0.5,0.7) to (0.2,1);
\draw[orange,thick,out = 90, in = 0, ->] (0.2,1) to (-0.5,1.25);

\draw[fill=black] (0,0.3) circle [radius=2pt];
\draw[fill=black] (0,1.5) circle [radius=2pt];
\draw[thick] (0,1) node {$\mathbf{\times}$};

\node[] at (-0.26,1.4) {$\alpha$};
\node[] at (0.26,1.41) {$\beta$};
\node[] at (-0.2,0.85) {$\tau$};

\end{tikzpicture}& 
\begin{tikzpicture}[scale = 1.7]

\tikzset{->-/.style={decoration={
  markings,
  mark=at position #1 with {\arrow{>}}},postaction={decorate}}}


\draw[out=30,in=-30,looseness=1.5] (0,0.3) to (0,1.5);
\draw[out=150,in=-150,looseness=1.5] (0,0.3) to (0,1.5);

\draw[in=0,out=45,looseness =1.25] (0,0.3) to (0,1.3);
\draw[in=180,out=135,looseness=1.25] (0,0.3) to (0,1.3);

\draw[orange,thick,out=-15,in=270] (-0.5,0.7) to (0.2,1);
\draw[orange, thick, out = 90, in = 90, looseness = 1.5] (0.2,1) to (-0.2,1);
\draw[orange, thick, out = 270, in = 270, looseness = 1.5] (-0.2,1) to (0.15,1);
\draw[orange, thick, out = 90, in = -15, ->] (0.15,1) to (-0.5,1.3);
\draw[fill=black] (0,0.3) circle [radius=2pt];
\draw[fill=black] (0,1.5) circle [radius=2pt];
\draw[thick] (0,1) node {$\mathbf{\times}$};

\node[] at (-0.26,1.4) {$\alpha$};
\node[] at (0.26,1.41) {$\beta$};
\node[] at (-0.2,0.85) {$\tau$};

\end{tikzpicture}&
\begin{tikzpicture}[scale = 1.7]

\tikzset{->-/.style={decoration={
  markings,
  mark=at position #1 with {\arrow{>}}},postaction={decorate}}}


\draw[out=30,in=-30,looseness=1.5] (0,0.3) to (0,1.5);
\draw[out=150,in=-150,looseness=1.5] (0,0.3) to (0,1.5);

\draw[in=0,out=45,looseness =1.25] (0,0.3) to (0,1.3);
\draw[in=180,out=135,looseness=1.25] (0,0.3) to (0,1.3);

\draw[orange,thick,out=-15,in=270] (-0.5,0.7) to (0.2,1);
\draw[orange, thick, out = 90, in = 90, looseness = 1.5] (0.2,1) to (-0.2,1);
\draw[orange, thick, out = 270, in = 195, ->] (-0.2,1) to (0.5,0.7);

\draw[fill=black] (0,0.3) circle [radius=2pt];
\draw[fill=black] (0,1.5) circle [radius=2pt];
\draw[thick] (0,1) node {$\mathbf{\times}$};

\node[] at (-0.26,1.4) {$\alpha$};
\node[] at (0.26,1.41) {$\beta$};
\node[] at (-0.2,0.85) {$\tau$};

\end{tikzpicture}\\   \hline  
\begin{tikzpicture}
\draw[gray,thick,dashed] (0,1) to node[midway,right,yshift=2pt,xshift=-2pt]{ $\alpha$} (1,0);
\draw[gray,thick,dashed] (1,1) to node[midway,right,yshift=2pt,xshift=-2pt]{ $\tau$} (2,0);
\draw[gray,thick,dashed] (1,2) to node[midway,right,yshift=2pt,xshift=-2pt]{ $\tau$} (2,1);
\draw[gray,thick,dashed] (2,2) to node[midway,right,yshift=2pt,xshift=-2pt]{ $\beta$} (3,1);
\draw[thick] (0,0) to (1,0) ;
\draw[thick] (1,0)to node[midway,right ,scale=1,yshift=-7pt, xshift = -2pt]{$x_\beta$} (1,1);
\node [above,scale=1,yshift=-2pt] at (0.5,1) {$x_\tau$};
\draw[thick] (1,1) to  (0,1) to (0,0);
\draw[thick] (1,0) to node[midway,below,yshift=2pt]{$x_\alpha$} (2,0) to node[right,yshift = -2pt]{$x_\tau$} (2,1) ;
\draw[thick] (1,2) to node[midway,above, yshift = -2pt]{$x_\beta$} (2,2);
\draw[thick] (2,2) to node[midway,right,xshift = -2pt,yshift=-7pt]{$x_\alpha$} (2,1) to node[midway,above,yshift=-2pt]{$x_\tau$} (1,1) to node[midway,left,xshift=3pt]{$x_\tau$} (1,2);
\draw[thick](2,1) to (3,1) to (3,2) to (2,2);
\node[below,yshift = 2pt] at (2.5,1){$x_\tau$};
\end{tikzpicture} & 
\begin{tikzpicture}
\draw[gray,thick,dashed] (0,1) to node[midway,right,yshift=2pt,xshift=-2pt]{ $\alpha$} (1,0);
\draw[gray,thick,dashed] (1,1) to node[midway,right,yshift=2pt,xshift=-2pt]{ $\tau$} (2,0);
\draw[gray,thick,dashed] (2,1) to node[midway,right,yshift=2pt,xshift=-2pt]{ $\tau$} (3,0);
\draw[gray,thick,dashed] (3,1) to node[midway,right,yshift=2pt,xshift=-2pt]{ $\beta$} (4,0);
\draw[thick] (0,0) to (1,0) to node[midway, below,yshift = 2pt]{$x_\alpha$} (2,0) to node[right,xshift = -2pt,yshift = -5pt]{$x_\tau$} (2,1) to node[midway,above,yshift=-2pt]{$x_\tau$} (1,1) to node[midway,above,yshift=-2pt]{$x_\tau$} (0,1) to (0,0);
\draw[thick] (2,0) to node[midway, below,yshift = 2pt]{$x_\tau$} (3,0) to node[midway, below,yshift = 2pt]{$x_\tau$} (4,0) to (4,1) to (3,1) to node[midway, above,yshift = -2pt]{$x_\beta$} (2,1);
\draw[thick] (1,0) to node[midway, right,xshift = -2pt, yshift = -5pt]{$x_\beta$}(1,1);
\draw[thick] (3,0) to node[midway, right,xshift = -2pt, yshift = -5pt]{$x_\alpha$} (3,1);
\end{tikzpicture}&
\begin{tikzpicture}
\draw[gray,thick,dashed] (0,1) to node[midway,right,yshift=2pt,xshift=-2pt]{ $\alpha$} (1,0);
\draw[gray,thick,dashed] (1,1) to node[midway,right,yshift=2pt,xshift=-2pt]{ $\tau$} (2,0);
\draw[gray,thick,dashed] (1,2) to node[midway,right,yshift=2pt,xshift=-2pt]{ $\tau$} (2,1);
\draw[gray,thick,dashed] (1,3) to node[midway,right,yshift=2pt,xshift=-2pt]{ $\alpha$} (2,2);
\draw[thick] (0,0) to (1,0) to node[midway,below,yshift = 2pt]{$x_\alpha$} (2,0) to node[midway, right,xshift = -2pt] {$x_\tau$} (2,1) to node[midway, right,xshift = -2pt] {$x_\alpha$} (2,2) to (2,3) to (1,3) to node[midway, left,xshift = 2pt] {$x_\tau$} (1,2) to node[midway, left,xshift = 2pt] {$x_\tau$} (1,1) to node[midway,above,yshift=-2pt]{$x_\tau$} (0,1) to (0,0);
\draw[thick] (1,0) to node[midway, right, yshift = -7pt, xshift = -2pt]{$x_\beta$} (1,1) to node[midway, above, xshift = -4pt, yshift = -2pt]{$x_\tau$} (2,1);
\draw[thick] (1,2) to node[midway, above, xshift = -4pt, yshift = -2pt]{$x_\beta$} (2,2);
\end{tikzpicture}&
\begin{tikzpicture}
\draw[gray,thick,dashed] (0,1) to node[midway,right,yshift=2pt,xshift=-2pt]{ $\alpha$} (1,0);
\draw[gray,thick,dashed] (1,1) to node[midway,right,yshift=2pt,xshift=-2pt]{ $\tau$} (2,0);
\draw[gray,thick,dashed] (2,1) to node[midway,right,yshift=2pt,xshift=-2pt]{ $\tau$} (3,0);
\draw[gray,thick,dashed] (2,2) to node[midway,right,yshift=2pt,xshift=-2pt]{ $\alpha$} (3,1);
\draw[thick] (0,0) to (1,0) to node[midway,below,yshift = 2pt]{$x_\alpha$} (2,0) to node[midway,below,yshift = 2pt]{$x_\tau$} (3,0) to node[midway, right,xshift = -2pt] {$x_\alpha$} (3,1) to (3,2) to (2,2) to node[midway, left,xshift = 2pt] {$x_\tau$} (2,1) to node[midway,above,yshift=-2pt]{$x_\tau$} (1,1) to node[midway,above,yshift=-2pt]{$x_\tau$} (0,1) to (0,0);
\draw[thick] (1,0) to node[midway, right, yshift = -7pt, xshift = -2pt]{$x_\beta$} (1,1);
\draw[thick] (2,0) to node[midway, right, yshift = -7pt, xshift = -2pt]{$x_\tau$} (2,1) to node[midway, above, yshift = -2pt, xshift = -4pt]{$x_\beta$} (3,1);
\end{tikzpicture} \\\hline
\end{tabular}}
\caption{Snake graph puzzle pieces for when an arc $\gamma$ crosses a pending arc two times consecutively. The arcs $\alpha$ and $\beta$ could be standard or pending arcs.} \label{table:SnakeGraphPuzzlePiece}
\end{table}
\end{center}

To construct the snake graph $\mathcal{G}_{\gamma,T}$, we will glue together tiles corresponding to arcs crossed consecutively by $\gamma$. If $\gamma$ crosses $\tau_i$ and $\tau_{i+1}$ consecutively, and $\tau_i$ and $\tau_{i+1}$ are distinct arcs, then these arcs form a triangle in the triangulation. Call the third arc in this triangle $\tau_{[i]}$. Then, we glue tiles $G_i$ and $G_{i+1}$ along the edge $\tau_{[i]}$; we see examples of this gluing in the puzzle pieces in Table \ref{table:SnakeGraphPuzzlePiece}. We use the appropriate planar embeddings so $\textrm{rel}(T,G_i) \neq \textrm{rel}(T,G_{i+1})$ and establish the convention that we use the positively oriented tile associated to $\tau_{i_1}$, i.e.,  $\textrm{rel}(G_1) = +1$. We also set the convention that we glue $G_{i+1}$ onto either the north or the east edge of $G_i$.

Now, consider a closed curve $\xi$ in $\mathcal{O}$ with triangulation $T = \{\tau_1,\ldots,\tau_n\}$. Again, pick a representative from the isotopy class of $\xi$ with minimal intersections with $T$ and minimal self intersections. Suppose this representative is not contractible and does not enclose a disk containing exactly one orbifold point. Pick a point $p$ on $\xi$ which is not a point in $\xi \cap T$ and is not inside a pending arc. Choose an orientation for $\xi$. Then construct a snake graph $\mathcal{G}_{\xi,T}$ for $\xi$, starting and ending at $p$ and using the chosen orientation. If $\tau_{i_1}$ is the first arc crossed by $\xi$ and $\tau_{i_d}$ is the last, with this chosen orientation and starting point, then $\tau_{i_1}$ and $\tau_{i_d}$ are distinct arcs which bound the triangle which contains $p$. Let $\tau'$ be the third arc in this triangle. We form the \emph{band graph} $\widetilde{\mathcal{G}}_{\xi,T}$ for $\xi$ by gluing the first and last tiles, $G_1$ and $G_d$, along the edges labeled by $\tau'$.

\subsection{Height Monomials and Expansion Formulas}

The snake graph expansion formula  uses the \emph{height monomial} of a snake graph to encode the $y$-variables. For both the surface and the orbifold case, these formulas work in the context of cluster algebras with principal coefficients.
To define a height monomial, we first give some definitions related to snake graphs. 

\begin{definition}\label{def:MinMaxMatch}
\begin{enumerate}
    \item Call an edge of a snake graph \emph{internal} if it borders two tiles. If an edge is not internal, call it a \emph{boundary} edge.
    \item A \emph{perfect matching} of a graph is a subset $P$ of the set of edges such that every vertex is incident to exactly one element of $P$.
    \item Any snake graph, $\mathcal{G}$, has exactly two matchings which only use boundary edges. One of these boundary matchings uses the Southern edge of the first tile. Call this matching the \emph{minimal matching}, $P_-$. Call the other matching with all boundary edges the \emph{maximal matching} $P_+$.
\end{enumerate}
\end{definition}

Given two perfect matchings $P$ and $Q$, let $P\ominus Q := (P\cup Q) \backslash (P \cap Q)$ be the symmetric difference. In \cite{Musiker-Schiffler}, Musiker and Schiffler describe the structure of  the symmetric difference between an arbitrary matching and a minimal matching.

\begin{lemma}[\cite{Musiker-Schiffler}]\label{lem:SymmDiff}
Let $\mathcal{G}$ be a snake graph, $P_-$ the minimal matching of $G$ and $P$ an arbitrary matching of $\mathcal{G}$. Then $P \ominus P_-$ is the set of boundary edges of a (possibly disconnected) subgraph $\mathcal{G}_P$ of $\mathcal{G}$. 
\end{lemma}

The structure of $P \ominus P_-$ inspires the definition of a \emph{height monomial}. For the following, we recall that tiles of a snake graph $\mathcal{G}_{\gamma,T}$ are labeled with arcs  crossed by $\gamma$.

\begin{definition}\label{def:HeightMonomial}
Let $P$ be a perfect matching of $\mathcal{G}_{\gamma,T}$, where $T = \{\tau_1,\ldots,\tau_n\}$, and let $\cup_{j \in J} G_j$ be the subgraph of $\mathcal{G}_{\gamma,T}$ enclosed by $P \ominus P_-$. We define the \emph{height vector} of the perfect matching $\mathbf{h}(P) = (h_1,\ldots,h_n) \in \mathbb{N}^n$ where $h_i$ is the number of times a tile with label $\tau_i$ appears in $\cup_{j \in J} G_j$. We also define the \emph{height monomial} $y(P) = y_1^{h_1}\cdots y_n^{h_n}$. 
\end{definition}

The definition of height monomial is generalized for the general orbifold case in \cite{banaian2020snake}, but we will not need this generalization when only working with orbifold points of order 3. 

In the following, given a triangulation $T$ of a surface $(S,M)$, we can build $Q^T$ in the same way as in Section \ref{sec:GenCAFromOrb}. Then, the cluster algebra associated to $(S,M)$ will be $\mathcal{A}(Q^T)$. In this cluster algebra, there are bijections between clusters and triangulations and between cluster variables and arcs \cite{fomin2007cluster}. We let $[x_\gamma]^T$ denote the cluster variable in $\mathcal{A}(Q^T)$ associated to the arc $\gamma$ written in terms of the cluster determined by $T$.

\begin{theorem}[\cite{MSW}]\label{thm:MSWClusterExpansion}
Let $(S,M)$ be a surface with triangulation $T$, and let $\gamma$ be an arc on $(S,M)$. Let $\text{cross}(\gamma,T)$ be the monomial $x_{\tau_1}^{a_1}\cdots x_{\tau_n}^{a_n}$ where $a_i$ is the number of distinct crossings of $\tau_i$ by $\gamma$.  Then, \[
[x_\gamma]^T = \frac{1}{\text{cross}(\gamma,T)} \sum_P \text{wt}(P) y(P)
\]
where we sum over all perfect matchings of $P$ and $\text{wt}(P)$ is the product of the weights of all edges in $P$.
\end{theorem}

\begin{theorem}[\cite{banaian2020snake}]\label{thm:EKClusterExpansion}
Theorem \ref{thm:MSWClusterExpansion} also holds for a generalized cluster algebra $\mathcal{A}(Q^T_{\text{gen}})$ from an unpunctured orbifold $\mathcal{O}$ with triangulation $T$. 
\end{theorem}

Now we turn to band graphs. Given a band graph $\widetilde{\mathcal{G}}_{\xi,T}$ let $\mathcal{G}_{\xi,T}$  be the snake graph on tiles $G_1,\ldots,G_d$ resulting from cutting the graph along an internal edge $e$. We refer to $e$ as the \emph{cut edge}. Let $x$ and $y$ be the endpoints of $e$ in $G_{1}$ and let $x'$ and $y'$ be the endpoints of $e$ in $G_{d}$ such that, in $\widetilde{\mathcal{G}}_{\xi,T}$, $x$ and $x'$ are identified and $y$ and $y'$ are identified. 

\begin{definition}\label{def:GoodMatching}
A perfect matching $P$ of a band graph $\widetilde{\mathcal{G}_{\xi,T}}$ is a \emph{good matching} if either the cut edge is in $P$ or if the edge of $P$ incident to $x = x'$ and the edge of $P$ incident to $y = y'$ are on the same side of the cut edge. 
\end{definition}

Good matchings of the band graph $\widetilde{\mathcal{G}}_{\xi,T}$ can be identified with the set of perfect matchings of $\mathcal{G}_{\xi,T}$ which include $(x,y)$ or $(x',y')$. In particular, when we form a band graph from a closed curve, the minimal and maximal matchings of the underlying snake graph $\mathcal{G}_{\xi,T}$ will give good matchings of $\widetilde{\mathcal{G}}_{\xi,T}$, so we can refer to the corresponding good matchings also as minimal and maximal. 

The weight monomial for a good matching is defined the same as for a perfect matching. The height monomial for a good matching is defined as the height monomial for the corresponding perfect matching of the underlying snake graph. With these definitions, we can define a generalized cluster algebra element $x_{\xi}$ associated to the closed curve $\xi$ by applying the expansion formula to the band graph $\widetilde{\mathcal{G}}_{\xi,T}$. This element will not be a cluster variable since we will never reach a closed curve by a sequence of flips starting from a triangulation. Similarly, we can apply the snake graph expansion formula to arcs which have self-intersection but these expressions will also not be cluster variables.

\subsection{Lattice of Perfect Matchings}\label{sec:LatticePMs}

The vectors $\mathbf{h}(P)$ associated to each perfect matching $P$ allow us to give a partial order to the set of perfect matchings. In the following, we will consider that each tile of a snake graph has a different label, so that $\mathbf{h}(P) = (h_1,\ldots,h_d)$ where $h_i = 1$ only if $G_i \in P_- \ominus P$. One can recover the usual height vector by combining the contributions from tiles which are associated to the same arc in $T$.

Given two distinct perfect matchings $P_1$ and $P_2$, we say $P_1 < P_2$ if $\mathbf{h}(P_1) < \mathbf{h}(P_2)$ in the usual partial order on $\mathbb{Z}^d$. To be precise, if $\mathbf{h}(P_1) = (h_1,\ldots,h_n)$ and $\mathbf{h}(P_2) = (h_1',\ldots,h_n')$, then $P_1 < P_2$ if $h_i \leq h_i'$ for all $i$, where at least one inequality is strict. Note that the minimal matching, $P_-$ has $\mathbf{h}(P_-) = (0,\ldots,0)$ and the maximal matching, $P_+$, has $\mathbf{h}(P_+) = (1,\ldots,1)$ so that these are minimal and maximal elements respectively in this poset.

 If $P_1 
\lessdot P_2$ so that the $\mathbf{h}(P_1)$ and $\mathbf{h}(P_2)$ only differ by 1 in one spot, then we know that $P_1$ and $P_2$ are related by a \emph{twist} at a tile. A twist is the result of taking a pair of edges in a perfect matching which border the same tile and replacing them with the other pair of edges along that tile. If it is possible to twist tile $G_i$ in a perfect matching $P$, then we will say $G_i$ is \emph{twistable}.

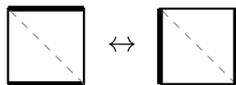
\begin{center}
\begin{tikzpicture}
\draw[thick] (0,0) -- (1,0) -- (1,1) -- (0,1) -- (0,0);
 \draw [line width=0.75mm] (0,0) -- (1,0);
\draw[line width = 0.75mm] (0,1) -- (1,1);
\draw[gray,dashed] (1,0) -- (0,1);
\node[] at (1.5,0.5){$\leftrightarrow$};
\draw[thick] (2,0) -- (3,0) -- (3,1) -- (2,1) -- (2,0);
\draw[gray,dashed] (3,0) -- (2,1);
 \draw [line width=0.75mm] (2,0) -- (2,1);
\draw[line width = 0.75mm] (3,0) -- (3,1);
\end{tikzpicture}
\end{center}

Musiker, Schiffler, and Williams show that this partial order is in fact a distributive lattice. In the following, we say three consecutive tiles in a snake graph, $G_{i-1}, G_i, G_{i+1}$ form a zig-zag if they share a common vertex and otherwise they form a straight line. For example, in the snake graph in Example \ref{ex:SnakeAndPoset}, tiles $G_1,G_2,G_3$ are in a straight line while tiles $G_2,G_3,G_4$ are in a zig-zag. Given a snake graph $\mathcal{G}$ on tiles $G_1,\ldots,G_d$, we define a partial order $\mathcal{P}_{\mathcal{G}}$ on $[d]$ with the following cover relations. If $G_2$ is glued on the north edge of $G_1$, we set $1 \lessdot 2$; otherwise, we set $1 \gtrdot 2$. Now, let $i \geq 2$. Suppose first that the three tiles $G_{i-1},G_i,G_{i+1}$ form a zig-zag shape. Then, if $i-1 \lessdot i$ we set $i \lessdot i+1$  and if $i-1 \gtrdot i$ we set $i \gtrdot i+1$.  Now suppose $G_{i-1},G_i,G_{i+1}$ lie in a straight line. Then,  if $i-1 \lessdot i$ ($i-1 \gtrdot i$), we set $i \gtrdot i+1$ ($i \lessdot i+1$). 

\begin{theorem}[Theorem 5.4 \cite{MSW-bases}]\label{Thm:PMPoset}
The poset on the set of perfect matchings of $\mathcal{G}_{\gamma,T}$ given by setting $P_1 < P_2$ whenever $\mathbf{h}(P_1) < \mathbf{h}(P_2)$ is a distributive lattice. Moreover, this lattice isomorphic to the lattice of order ideals of $\mathcal{P}_{\mathcal{G}_{\gamma,T}}$. 
\end{theorem}

Now, consider a band graph $\widetilde{\mathcal{G}}$. Cut this band graph along one internal edge to make a snake graph $\mathcal{G}$ on tiles $G_1,\ldots,G_d$, and let $\mathcal{P}_{\mathcal{G}}$ be the resulting quiver on $[d]$. The minimal matching, $P_-$ of $\mathcal{G}$ will use exactly one of the cut edges. If it uses the cut edge on $G_1$, we form $\mathcal{P}_{\widetilde{\mathcal{G}}}$ by adding the arrow $1 \to d$ to $\mathcal{P}_{\mathcal{G}}$. Otherwise, $P_-$ uses the cut edge on $G_d$ and $\mathcal{P}_{\widetilde{\mathcal{G}}}$ is the result of adding the arrow $d \to 1$ to $\mathcal{P}_{\mathcal{G}}$. One can check that the choice of cut edge does not affect the quiver $\mathcal{P}_{\widetilde{\mathcal{G}}}$.

\begin{theorem}[Theorem 5.7 \cite{MSW-bases}] \label{thm:GMPoset}
The poset on the set of good matchings of $\widetilde{\mathcal{G}}_{\xi,T}$ given by setting $P_1 < P_2$ whenever $\mathbf{h}(P_1) < \mathbf{h}(P_2)$ is a distributive lattice. Moreover, this lattice isomorphic to the lattice of order ideals of $\mathcal{P}_{\widetilde{\mathcal{G}}}$. 
\end{theorem}

\begin{example}\label{ex:SnakeAndPoset}
First, we give a snake graph $\mathcal{G}$ and the Hasse diagram of the corresponding poset $\mathcal{P}_{\mathcal{G}}$.
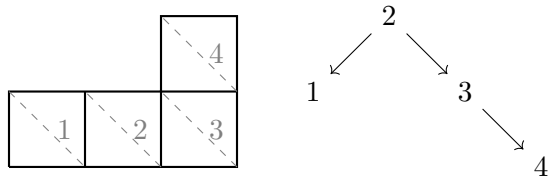
\begin{center}
\begin{tikzpicture}
\draw[thick] (0,0) -- (3,0) -- (3,2) -- (2,2) -- (2,1) -- (0,1) -- (0,0);
\draw[thick] (1,0) -- (1,1);
\draw[thick] (2,0) to (2,1) to (3,1);
\draw[gray,dashed] (0,1) to node[midway,right]{$1$} (1,0);
\draw[gray,dashed] (2,0) to node[midway,right]{$2$} (1,1);
\draw[gray,dashed] (3,0) to node[midway,right]{$3$} (2,1);
\draw[gray,dashed] (3,1) to node[midway,right]{$4$} (2,2);
\node(1) at (4,1){$1$};
\node(2) at (5,2){$2$};
\node(3) at (6,1){$3$};
\node(4) at (7,0){$4$};
\draw[->] (2) to (1);
\draw[->] (2) to (3);
\draw[->] (3) to (4);
\end{tikzpicture}    
\end{center}

Then, we give the lattice of perfect matchings of $\mathcal{G}$, where we represent a perfect matching $P$ with its height vector $\mathbf{h}(P)$. The bijection between height vectors and order ideals of $\mathcal{P}_{\mathcal{G}}$ is clear.

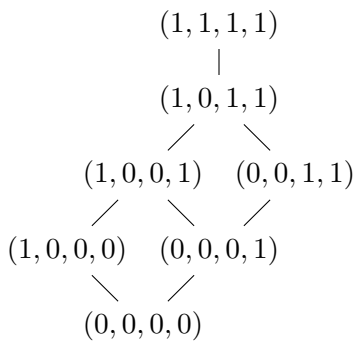
\begin{center}
\begin{tikzpicture}
\node(0000) at (0,0){$(0,0,0,0)$};
\node(1000) at (-1,1){$(1,0,0,0)$};
\node(0001) at (1,1){$(0,0,0,1)$};
\node(1001) at (0,2){$(1,0,0,1)$};
\node(0011) at (2,2){$(0,0,1,1)$};
\node(1011) at (1,3){$(1,0,1,1)$};
\node(1111) at (1,4){$(1,1,1,1)$};
\draw[] (0000) to (1000);
\draw[] (0000) to (0001);
\draw[] (1000) to (1001);
\draw[] (0001) to (1001);
\draw[] (0001) to (0011);
\draw[] (0011) to (1011);
\draw[] (1001) to (1011);
\draw[] (1011) to (1111);
\end{tikzpicture}    
\end{center}

We could glue the south edge of $G_1$ to the north edge of $G_4$ to form a band graph $\widetilde{\mathcal{G}}$. Then the lattice of height vectors of good matchings would be the same as above except with the vector $(1,0,0,0)$ removed. Notice that $P_-$ uses the south edge of $G_1$ by convention. Therefore, the resulting quiver $\mathcal{P}_{\widetilde{\mathcal{G}}}$  would add the edge $4 \to 1$. Accordingly, we can see that the set $\{1\}$ would no longer be an order ideal as 1 would no longer be minimal. \end{example}

Musiker, Schiffler, and Williams also prove the following result, called the ``twist-parity condition'', which describes how the parity of the index of a tile affects what type of twist moves up the lattice of perfect matchings. This will be useful in some proofs in Section \ref{sec:CCMap}.

\begin{lemma}[\cite{MSW-bases}]\label{lem:twistparity}
Let $\mathcal{G}$ be a snake graph, and let $L(\mathcal{G})$ be its lattice of perfect matchings. If $i$ is odd (even), then a twist on $G_i$ which replaces horizontal edges with vertical (vertical edges with horizontal) moves up in $L(\mathcal{G})$. 
\end{lemma}

The same statement is true in a band graph associated to an arc $\widetilde{\mathcal{G}}_{\xi,T}$ where we replace the role of odd and even indexed tiles with tiles with orientation $+1$ and $-1$ respectively.

\section{Gentle Algebras}\label{sec:Strings}

In this section, we begin by reviewing concepts and notations associated gentle algebras and their representations. We direct the interested reader to \cite{schiffler2014quiver} or \cite{assem2006elements} for many more details.

Let $K$ be an algebraically closed field and $Q=(Q_0, Q_1)$ be a finite quiver, not necessarily a cluster quiver. 
A \emph{path} $w$ of length $m$ in $Q$ is a sequence of arrows $\alpha_1\cdots \alpha_m$ such that $t(\alpha_k) = s(\alpha_{k+1})$ for each $k=1, \dots, m-1$, and we denote by $s(w)=s(\alpha_1)$ and by $t(w)=t(\alpha_n)$. For each vertex $v\in Q_0$, we denote by $e_v$ the \emph{trivial path} of length 0.

The \emph{path algebra} $KQ$ is defined as the $K$-vector space with basis the set of all paths in $Q$, and the multiplication is induced by the concatenation of paths. Denote by $R$ the ideal generated by the arrows of $Q$. An ideal $I$ of $KQ$ is called \emph{admissible} if there is an natural number $m\geq 2 $ such that  $R^m\subseteq I\subseteq ^2$. The pair $(Q,I)$ is called \emph{bounded quiver}.

It is shown in \cite{assem2006elements} that module category of $KQ/I$ is equivalent to category of \emph{bound quiver representations} $(M_i, \phi_\alpha)_{i\in Q_0, \alpha\in Q_1}$ of the bound quiver $(Q,I)$, where $M_i$ is a $K$-vector space and $\phi_{\alpha}:M_i\to M_j$ a $K$ linear map for every arrow $\alpha:i\to j$. This latter category is denoted $\text{rep}(Q,I)$. Given any path $w =\alpha_1 \dots \alpha_m$  in $Q$ we have the induced map $\phi_w =\phi_{\alpha_m}\dots \phi_{\alpha_1}$. This definition is extended by linearity to any linear combination of paths $\rho=\sum_{i=1}^{k}\lambda_iw_i$; in this case, the induced map is denoted by $\phi_{\rho}$. A pair $(M_i, \phi_\alpha)_{i\in Q_0, \alpha\in Q_1}$ is a representation of the bound quiver $(Q, I)$ if $\phi_{\rho} = 0$ for all $\rho\in I$.

In this work, the class of gentle algebras, whose representation theory is particularly well understood,   plays an important role. In particular, in Lemma \ref{lem:QuiverIsGentle}, we will show that the algebras we will consider will always be \emph{gentle algebras}. In the following, we review some basic definition of gentle algebras and their representations. 

\begin{definition}\label{def:GentleAlgebra}
An algebra $A$ is \emph{gentle} if it is morita equivalent to a finite dimensional algebra $\Lambda=KQ/I$ such that the following conditions hold.
\begin{enumerate}
    \item Each $v \in Q_0$ has at most two incoming and at most two outgoing arrows.
    \item For each $\beta \in Q_1$, there is at most one arrow $\alpha_1 \in Q_1$ such that $\alpha_1\beta$ is a path and $\alpha_1\beta \notin I$ and at most one arrow $\alpha_2 \in Q_1$ such that $\alpha_2\beta$ is a path and $\alpha_2\beta \in I$.
    \item For each $\beta \in Q_1$, there is at most one arrow $\gamma_1 \in Q_1$ such that $\beta \gamma_1$ is a path and $\beta \gamma_1 \notin I$ and at most one arrow $\gamma_2 \in Q_1$ such that $\beta \gamma_2$ is a path and $\beta\gamma_2 \in I$.
    \item The ideal $I$ is generated by paths of length two.
\end{enumerate}
\end{definition}

If a quiver $Q$ and an ideal $I$ of the path algebra $KQ$ satisfy the above definition, we will also say that $(Q,I)$ is a \emph{gentle pair}. 

\subsection{Strings and Bands}

In this section, we review the definition of \emph{strings} and \emph{bands} that describe combinatorially the representations of gentle algebras.

For every $\alpha \in Q_1$, we define a formal inverse $\alpha^{-1}$ where $s(\alpha) = t(\alpha^{-1})$ and $t(\alpha) = s(\alpha^{-1})$. Given a path $w = \alpha_1 \cdots \alpha_m$, we define $w^{-1}$ as $\alpha_m^{-1}\cdots \alpha_1^{-1}$. Denote by $Q_1^{-1}$ the set of formal inverse arrows of $Q$. A walk $w$ of length $m$ is a word or sequence $\alpha_1 \dots \alpha_m$ of elements of $Q_1\cup Q_1^{-1}$ such that $t(\alpha_i)=s(\alpha_{i+1})$  and $\alpha_i$ is not the formal inverse of $ \alpha_{i+1}$ for every $i=1, \dots m-1$. Given a walk $w=\alpha_1 \dots \alpha_m$, we denote by $s(w)=s(\alpha_1)$ and $t(w)=t(\alpha_m)$. We say that a walk is trivial if it is of length 0. By abuse of notation, we write $(\alpha^{-1})^{-1}=\alpha$ for every arrow $\alpha$ of $Q_1$.

\begin{definition}
Let $(Q,I)$ be a gentle pair.
A \emph{string} in $(Q,I)$ is a walk $w = \alpha_1\dots \alpha_m$ such that there is no subword of $w$ which is a relation of $I$ or subword whose inverse is a relation in $I$. 
A string $w=\alpha_1\dots\alpha_m$ is \emph{direct} if each $\alpha_i$ is an arrow of $Q_1$. 
A string $w=\alpha_1\dots\alpha_m$ is \emph{inverse} if each $\alpha_i$ is an arrow of $Q_1^{-1}$.

We also have a trivial string $e_v$ for each $v \in Q_0$.

A non trivial walk $\overline{w}$ is a \emph{band} if any power of $w$ is a sting and $w$ is not the power of any string of a smaller length.
\end{definition}

Given a string $w = \alpha_1 \cdots \alpha_m$, for any $i \in [m]$, let $w(i) = s(\alpha_i)$ and set $w(m+1) = t(\alpha_m)$. We call $w(i) \in Q_0$ a \emph{vertex of $w$}. We also use this language for bands. The same vertex in $Q_0$ can appear multiple times as a vertex of $w$.

Given a string $w=\alpha_1\dots \alpha_m$, the string module $M(w)$ is obtained by
replacing every vertex of $w$ by a copy of the field $K$, every direct arrow by the identity map, and for inverse arrows the identity map in the opposite direction.

In a similar way, we can define modules from bands, but in this case, any band $\overline{w}=\alpha_1\dots \alpha_m$ induces a class of modules, called \emph{band modules}, $M(\overline{w}, n, \lambda)$, where $n$ is a natural number and $\lambda\in K$ a scalar different from zero. The band module $M(\overline{w}, n, \lambda)$ is obtained by $\overline{w}$ by replacing each vertex $i=1, \dots, m-1$ by a copy of the $n$-dimensional vector space $K^n$, every direct arrow $\alpha_i$  with $i=1, \dots, m-1$ by the identity matrix of dimension $n$ and every inverse arrow by the identity matrix of dimension $n$ in the opposite direction, and $\alpha_m$ is replaced by the Jordan block of dimension $n$ and eigenvalue $\lambda$. Later, we will only consider the case where $n = 1$, so the Jordan block will just be multiplication by $\lambda$.

The following Theorem show that the set of strings and bands of a gentle pair $(Q,I)$ encode the set of indecomposable $KQ/I$ modules.

\begin{theorem}[\cite{butler1987auslander}]\label{thm:BRString}
Let $\Lambda = KQ/I$ be a gentle algebra. Then the set strings modules and band modules form a complete list of indecomposable $\Lambda$-modules. Moreover,  
\begin{itemize}
    \item string modules $M(w)$ and $M(w')$ are isomorphic if and only if $w' = w^{\pm 1}$,  and
    \item band modules $M(\overline{w},\lambda,n)$ and $M(\overline{w}',\lambda',n')$ are isomorphic if and only if $\lambda=\lambda', n = n',$ and $\overline{w}$ and $\overline{w}'$ are equivalent up to cyclic permutation. 
\end{itemize}
\end{theorem}

It is common and convenient to draw strings as orientations of a type $A$ Dynkin diagram, embedded in the plane so that draw direct arrows go down and to the right and inverse arrows go down and to the left. One can similarly draw bands as orientations of a type $\widetilde{A}$ Dynkin diagram.

\begin{example}\label{Ex:Strings}

Consider the following quiver $Q$. The ideal $I = \langle \alpha\beta, \beta \gamma, \gamma \alpha \rangle $ is an admissible ideal of $I$. One can check that $(Q,I)$ is a gentle pair.

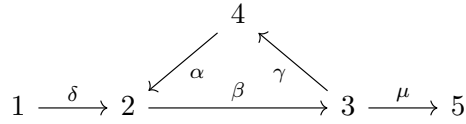
\begin{center}
\begin{tikzcd}
& & 4 \arrow[dl, "\alpha"] & & \\ 
1 \arrow[r, "\delta"] & 2 \arrow[rr,"\beta"] & & 3 \arrow[ul,"\gamma"]  \arrow[r,"\mu"] & 5\\
\end{tikzcd}
\end{center}

Some strings on $Q$ with ideal $I$ are $\alpha \delta^{-1}, \delta \beta \mu,$ and $\mu^{-1}\gamma$. There are no bands. We draw these strings below.

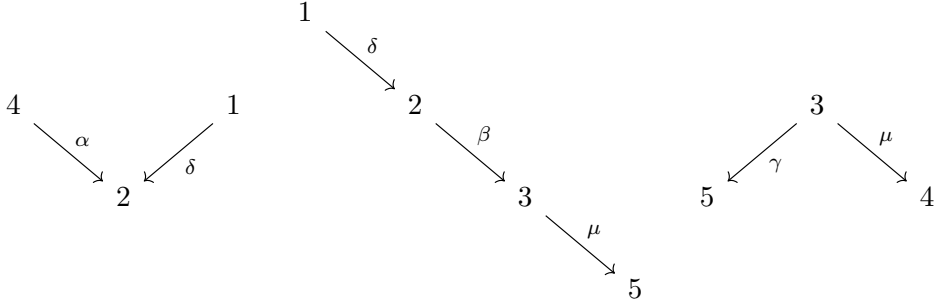
\begin{tabular}{ccc}
\begin{tikzcd}
4\arrow[dr, "\alpha"] & & 1 \arrow[dl, "\delta"] \\
& 2 & \\
\end{tikzcd} & 
\begin{tikzcd}
1 \arrow[dr, "\delta"] & & & \\
& 2 \arrow[dr,"\beta"] & & \\
& & 3 \arrow[dr, "\mu"] & \\
& & & 5 \\
\end{tikzcd}
& 
\begin{tikzcd}
& 3 \arrow[dr, "\mu"] \arrow[dl, "\gamma"] & \\
5 & & 4\\
\end{tikzcd}
\end{tabular}

\end{example}

We provide a few more definitions related to the pictorial representation of a string; these  will be useful in Section \ref{sec:CCMap}.

\begin{definition}\label{def:TopAndBottom}
 Let $w = \alpha_1\cdots \alpha_m$ be a string for the gentle pair $(Q,I)$. We say that a vertex $v$ of the string is a \emph{peak} if either
 \begin{enumerate}
     \item[(1)] $v=s(w)$ and $\alpha_1$ is a direct letter,
     \item[(2)] $v=t(w)$ and $\alpha_m$ is an inverse letter, or 
     \item[(3)] $v$ is start of $\alpha_i$ and $\alpha_{i+1}$ for some $i$ with $1 \leq i \leq m-1$.
 \end{enumerate}
 
 The definition of a \emph{deep} is dual. In each case, we say $v$ is \emph{strict} if it only satisfies condition (3).
 
 If $\overline{w}$ is a band, then we define peaks and deeps in the same way; consequently, every peak or deep is strict.
\end{definition}

\begin{definition}\label{def:PeakAndDeep}
Let $w$ be a string.
\begin{itemize}
    \item We say $w$ \emph{starts (ends) on a peak} if there does not exist $\alpha \in Q_1$ such that $\alpha w$ ($w \alpha^{-1}$) is a string.
    \item We say $w$ \emph{starts (ends) on a deep} if there does not exist $\alpha \in Q_1$ such that $\alpha^{-1}w$ ($w \alpha$) is a string.
\end{itemize}
\end{definition}

For example, given the admissible pair $(Q,I)$ in Example 
\ref{Ex:Strings}, the direct string  $\alpha$ starts on a peak and both starts and ends on a deep. It does not end on a peak before $\alpha \delta^{-1}$ is a string.
These definitions again have a visual expectation. For example, $w$ starts on a peak if it is not possible to extend $w$ on the left with an arrow that goes higher. 

We now introduce hooks and cohooks. These notions will be useful to describe irreducible morphisms, and as a consequence the Auslander-Reiten translation at a string module in terms on the combinatorics of the string. First, we describe adding hooks to either end of a string. 

\begin{definition}\label{def:AddHookCohook}
Let $w$ be a string from a gentle pair $(Q,I)$.
\begin{enumerate}
    \item Suppose that $w$ is a string which does not start on a peak, that is, there exists $\beta$ such that $\beta w$ is a string. Let $w'$ be a maximal inverse string such that $t(w')=s(\beta)$. Define ${}_hw$ the string $w'\beta w $. We say that ${}_hw$ is obtained from $w$ by adding a hook on $s(w)$.

    \item Suppose $w$ is a string which does not end on a peak, that is, there exists $\beta^{-1}$ such that $w\beta^{-1}$ is a string. Let $w'$ be the maximal direct string such that $t(\beta^{-1})=s(w')$. Denote by $w_{h}$ the string $w\beta^{-1} w' $. We say that $w_h$ is obtained from $w$ by adding a hook on $t(w)$.
 
\end{enumerate}
\end{definition}

The definition of adding cohooks is dual, and it is denoted by ${}_cw$ and $w_{c}$ the strings obtained from $w$ by adding a cohook on $s(w)$ or $t(w)$ respectively.
Note that for each item in Definition \ref{def:AddHookCohook}, it is possible that $w'$ is a trivial string, so that adding a hook or cohook only involves adding a single arrow to $w$. The requirements related to starting/ending on a peak or deep ensures that $\beta$ exists in each case. 

For most strings, we also have a notion of removing hooks and cohooks from the beginning and end. We again only provide details for hooks as the cohook cases are dual. 

\begin{definition}\label{def:RemoveHookCohook}
Throughout, let $w = \alpha_1\cdots \alpha_m$ be a string.
If $w$ is not an inverse string, let $j\leq m$ be the smallest index such that $\alpha_j$ is direct. Then, we define $_{h^{-1}}w = \alpha_{j+1} \cdots \alpha_m$ where if $j = m$, then we define $_{h^{-1}}w = e_{t(\alpha_m)}$. If $w$ is an inverse string, then  we define  $_{h^{-1}}w = 0$. Define $w_{h^{-1}}$ dually.
\end{definition}

    In the following sections, we need to emphasize if a letter of a string is direct or inverse, to do that we write a letter $\beta$ of a string $w$ as $\alpha^\epsilon$, where $\epsilon=+$ describes a direct letter and $\epsilon=-$ describes an inverse letter.

\section{Module Category from Triangulated Orbifold}\label{sec:ModuleCategory}

 We begin by building a gentle algebra $KQ_T/I_T$ from an orbifold $\mathcal{O}$ with triangulation $T$. Then, we use formulas related to string and band modules to relate $\text{rep}(Q_T,I_T)$ to the combinatorics of arcs on $\mathcal{O}$, inspired by similar results for the surface case in \cite{brustle2011cluster}.

\subsection{Modules from Arcs}

In Section \ref{sec:GenCAFromOrb}, we described how to produce a generalized cluster quiver $Q_{gen}^T$ from a triangulation $T$ of an orbifold $\mathcal{O}$. Now, given $T$ we build a related quiver $Q_T$ along with an admissible ideal $I_T$ of $KQ_T$. To form $Q_T$, we take $Q_{gen}^T$ and at each vertex $v \in P \subseteq Q_0$, we include a loop. We then forget the partitioning $Q_0 = S \sqcup P$ which was necessary for the definition of the generalized cluster quiver. The ideal $I_T$ is generated by paths of length two either of the form (1)  $\alpha\beta$ such that the arcs $\tau_{s(\alpha)}, \tau_{t(\alpha)} = \tau_{s(\beta)}$ and $\tau_{t(\beta)}$ form a triangle or (2) $\delta^2$ where $\tau_{s(\delta)}$ is a pending arc (implying that $\delta$ is a loop). See Table \ref{table:PuzzlePieces} for the quiver and relations corresponding to each type of triangle in a triangulated unpunctured orbifold. A general quiver with relations from a triangulation $T$ is the result of gluing these pieces together along vertices from standard arcs (that is, vertices without an incident loop) and taking the union of the generators of the ideal from each piece. These puzzle pieces were also given in \cite{FST-orbTriang}, where orbifold points of different orders are considered.

\begin{table}
\begin{tabular}{c|c|c|c}
  & Type 0 & Type 1 & Type 2\\
    $T$ & \begin{tikzpicture}[scale=1.25]
     \node[above] at (0.5,0.5){2};
     \draw(1,1) to (0,0) to (2,0) to (1,1);
     \node[below] at (1,0) {1};
     \node[above] at (1.5,0.5){3};
    \end{tikzpicture} & \begin{tikzpicture}[scale = 1.25]
\draw[out=30,in=-30,looseness=1.5] (0,0.3) to (0,2.5);
\draw[out=150,in=-150,looseness=1.5] (0,0.3) to (0,2.5);
\draw[in=0,out=45,looseness =1.25] (0,0.3) to (0,2);
\draw[in=180,out=135,looseness=1.25] (0,0.3) to (0,2);
\draw[thick] (0,1.7) node {$\mathbf{\times}$};
\node[right] at (0.8,1.5){$\tau_3$};
\node[left] at (-0.8,1.5){$\tau_2$};
\node[above] at (0,2){$\tau_1$};
\end{tikzpicture} &
\begin{tikzpicture}[scale=2]
\draw[out=75,in=-180] (0,0) to (0.3,0.85);
\draw[out=0,in=0, looseness = 1] (0,0) to node[midway,left,yshift=-4pt]{$2$} (0.3,0.85);
\draw[out=105,in=0] (0,0) to (-0.3,0.85);
\draw[out=180,in=180, looseness = 1] (0,0) to node[midway,right,yshift=-4pt]{$1$} (-0.3,0.85);
\draw[looseness=1.75,out=0,in=0] (0,0) to (0,1.3);
\draw[looseness=1.75,out=180,in=180] (0,1.3) to (0,0);
\draw[thick] (-0.3,0.65) node {$\mathbf{\times}$};
\draw[thick] (0.3,0.65) node {$\mathbf{\times}$};
\node [yshift=8pt] at (0,1) {$3$};
\end{tikzpicture}
\\ \hline
   $Q^T_{gen}$  & 
\begin{tikzcd}[arrow style=tikz,>=stealth,row sep=4em]
2 \arrow[dr,"\alpha"]
&& 3 \arrow[ll,"\beta"]  \\
& 1 \arrow[ur,"\gamma"]
\end{tikzcd} &  
\begin{tikzcd}[arrow style=tikz,>=stealth,row sep=4em]
2 \arrow[dr,"\alpha"]
&& 3 \arrow[ll,"\beta"]  \\
& \bar{1} \arrow[ur,"\gamma"]
\end{tikzcd}& 
\begin{tikzcd}[arrow style=tikz,>=stealth,row sep=4em]
\bar{2} \arrow[dr,"\alpha"]
&& 3 \arrow[ll,"\beta"]  \\
& \bar{1} \arrow[ur,"\gamma"]
\end{tikzcd}
\\
$Q_T$& \begin{tikzcd}[arrow style=tikz,>=stealth,row sep=4em]
2 \arrow[dr,"\alpha"]
&& 3 \arrow[ll,"\beta"]  \\
& 1 \arrow[ur,"\gamma"]
\end{tikzcd} &  
\begin{tikzcd}[arrow style=tikz,>=stealth,row sep=4em]
2 \arrow[dr,"\alpha"]
&& 3 \arrow[ll,"\beta"]  \\
& 1 \arrow[ur,"\gamma"]
\arrow[out=-90,in=0,loop, "\epsilon"]
\end{tikzcd}&
\begin{tikzcd}[arrow style=tikz,>=stealth,row sep=4em]
2 \arrow[dr,"\alpha"]
\arrow[out=90,in=180,loop, "\rho"]
&& 3 \arrow[ll,"\beta"]  \\
& 1 \arrow[ur,"\gamma"]
\arrow[out=-90,in=0,loop, "\epsilon"]
\end{tikzcd}\\
$I_T$ & $\langle \alpha\gamma, \gamma \beta, \beta \alpha \rangle$& $\langle \alpha\gamma, \gamma \beta, \beta \alpha, \epsilon^2 \rangle$& $\langle \alpha\gamma, \gamma \beta, \beta \alpha, \epsilon^2, \rho^2 \rangle$\\
\end{tabular}
\caption{The first row gives types of triangles in a triangulation of an unpunctured orbifold. The second row gives the pieces of a generalized cluster quiver $Q_{gen}^T$ arising from each triangle. The bottom rows give the quiver $Q_T$ and relations $I_T$ from each triangle. }\label{table:PuzzlePieces}
\end{table}

These algebras from triangulated orbifolds will always be gentle. They are studied in \cite{labardini2023gentleI} and \cite{labardini2023gentle}.

\begin{lemma}[Lemma 4.2 \cite{labardini2023gentleI}]\label{lem:QuiverIsGentle}
The algebra $KQ_T/I_T$ is gentle.
\end{lemma}

Given an arc $\gamma$ in an orbifold $\mathcal{O}$ with triangulation $T = \{\tau_1,\ldots,\tau_n\}$ such that $\gamma \notin T$,  we can construct a string $w_\gamma$ in $KQ_T/I_T$.  Pick an orientation for $\gamma$, and let $\tau_{i_1},\ldots,\tau_{i_{m+1}}$ be the arcs in $T$ crossed by $\gamma$.  We will establish a string $w_{\gamma} = \alpha_1^{\epsilon_1}\cdots \alpha_m^{\epsilon_m}$. If $m =0$, then $w_\gamma = e_{\tau_{i_1}}$. Now, suppose $m > 0$.  Necessarily, for all $1 \leq j \leq m$, there is an arrow $\alpha_j$ between $i_j$ and $i_{j+1}$ in $Q_T$.  If $i_j \neq i_{j+1}$ and $\tau_{i_{j+1}}$ follows $\tau_{i_j}$ counterclockwise,  we pick $\epsilon_j = +$; if instead $\tau_{i_{j+1}}$ follows $\tau_{i_j}$ clockwise, we pick $\epsilon_j = -$.  If $i_j =i_{j+1}$,  then $\tau_{i_j}$ is a pending arc.  Then we set $\alpha_j$ to be the loop at vertex $i_j$ in $Q_T$, and we set $\epsilon_j = +$ only if, between its intersections with $\tau_{i_j}$,  the unique marked point incident to $\tau_j$ is to the right of $\gamma$ (or equivalently, if the orbifold point enclosed in $\tau_j$ is to the left of $\gamma$). 

Similarly, given a closed curve $\xi$ which is not contractible and does not enclose a disc with one orbifold point, we can build a band $\overline{w}_\xi$. As for building a band graph for $\xi$ in Section \ref{subsec:SnakeGraphsfromOrbifold}, we choose a point $p$ on $\xi$ which is not an intersection point of $\xi$ and $T$ nor inside a pending arc. We also choose an orientation of $\xi$. Using the chosen orientation for $\xi$ with start and end point at $p$, let $\tau_{i_1},\cdots,\tau_{i_m}$ be the ordered list of arcs crossed by $\xi$. First, we build a string $w_\xi = \alpha_1^{\epsilon_1}\cdots \alpha_{m-1}^{\epsilon_{m-1}}$ as before from this list of arcs. Necessarily since $\tau_{i_1}$ and $\tau_{i_m}$ border the same triangle since this is the triangle which contains $p$. There must be $\alpha_m \in Q_T$ and $\epsilon_m \in \{+,-\}$, such that $s(\alpha_m^{\epsilon_m}) =i_m$ and $t(\alpha_{m}^{\epsilon_m}) = i_1$. Then we set $\overline{w}_\xi = \alpha_1^{\epislon_1} \cdots \alpha_{m-1}^{\epsilon_{m-1}}\alpha_m^{\epsilon_m}$. 

We claim that every string or band in $KQ_T/I_T$ corresponds to an arc or closed curve in this way.

\begin{theorem}\label{thm:StringsBijectCurves}
Fix an orbifold $\mathcal{O}$ with triangulation $T$.
\begin{enumerate}
    \item Each string in $KQ_T/I_T$ corresponds to a (possibly generalized) arc in $\mathcal{O}$. 
    \item Each band in $KQ_T/I_T$ corresopnds to a closed curve in $\mathcal{O}$. 
\end{enumerate}
In each case, the corresponding curve is unique up to isotopy.
\end{theorem}

\begin{proof}

If $w = e_v$ for $v \in Q_0$, then the associated arc will be the other diagonal in the quadrilateral around $\tau_v$ in $T$. Now, consider a string $w = \alpha_1^{\epsilon_1}\cdots \alpha_m^{\epsilon_m}$. We can consider $Q_T$ as a combination of pieces as in Table \ref{table:PuzzlePieces}. Given two consecutive arrows in $w$, $\alpha_j$ and $\alpha_{j+1}$ for $1 \leq j < m$, these arrows are either in different triangle pieces of $Q_T$ or exactly one of these arrows is a loop.  In the latter case, if $\alpha_j$ is a loop then $\alpha_{j+1}$ must be in the same triangle piece and same for $\alpha_{j-1}$, since we form these quivers by gluing vertices which are not incident to loops. 

Using this fact, we piece $\gamma$ together with small segments in $\mathcal{O}$.  For each factor $\alpha_i^{\epsilon_i}$ in $w$  for a non-loop arrow $\alpha_i$, we draw a segment $\rho_i$ between the arcs $\tau_{t(\alpha_i)}$ and $\tau_{s(\alpha_i)}$. If $\epsilon_i = +$, we orient $\rho_i$ from $\tau_{s(\alpha_i)}$ to $\tau_{t(\alpha_i)}$; otherwise, we orient $\rho_i$ from $\tau_{t(\alpha_i)}$ to $\tau_{s(\alpha_i)}$.

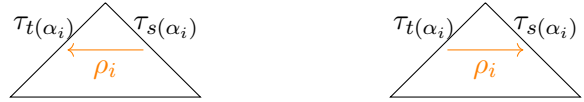
\begin{center}
\begin{tikzpicture}[scale=1.25]
     \node[above,  xshift = -5] at (0.5,0.5){$\tau_{t(\alpha_i)}$};
     \draw(1,1) to (0,0) to (2,0) to (1,1);
     \node[above, xshift = 5] at (1.5,0.5){$\tau_{s(\alpha_i)}$};
     \draw[orange, ->] (1.4,0.5) to node[below]{$\rho_i$} (0.6,0.5);    
     \node[above,  xshift = -5] at (4.5,0.5){$\tau_{t(\alpha_i)}$};
     \draw(5,1) to (4,0) to (6,0) to (5,1);
     \node[above, xshift = 5] at (5.5,0.5){$\tau_{s(\alpha_i)}$};
     \draw[orange, <-] (5.4,0.5) to node[below]{$\rho_i$} (4.6,0.5);
    \end{tikzpicture} 
 \end{center}   

If $\alpha_i$ is a loop,  then $\tau_{s(\alpha_i)} = \tau_{t(\alpha_i)}$ is a pending arc.  We draw a segment $\rho_i$ inside the pending arc.  If $\epsilon_i = +$, we orient $\rho_i$ so that the unique marked point incident to $\tau_{s(\alpha_i)}$ is to the right of $\rho_i$; otherwise, we orient $\rho_i$ so that this marked point is to the left.  

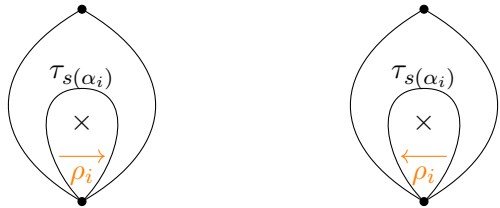
\begin{center}
\begin{tikzpicture}[scale = 1.5]
\draw[out=30,in=-30,looseness=1.5] (0,0.3) to (0,2);
\draw[out=150,in=-150,looseness=1.5] (0,0.3) to (0,2);
\draw[in=0,out=45,looseness =1.25] (0,0.3) to (0,1.3);
\draw[in=180,out=135,looseness=1.25] (0,0.3) to (0,1.3);
\draw[thick] (0,1) node {$\mathbf{\times}$};
\draw[fill=black] (0,0.3) circle [radius=1pt];
\draw[fill=black] (0,2) circle [radius=1pt];
\node[] at (0,1.4){$\tau_{s(\alpha_i)}$};
\draw[->, orange] (-0.2,0.7) to node[below]{$\rho_i$} (0.2,0.7);
\draw[out=30,in=-30,looseness=1.5] (3,0.3) to (3,2);
\draw[out=150,in=-150,looseness=1.5] (3,0.3) to (3,2);
\draw[in=0,out=45,looseness =1.25] (3,0.3) to (3,1.3);
\draw[in=180,out=135,looseness=1.25] (3,0.3) to (3,1.3);
\draw[thick] (3,1) node {$\mathbf{\times}$};
\draw[fill=black] (3,0.3) circle [radius=1pt];
\draw[fill=black] (3,2) circle [radius=1pt];
\draw[<-, orange] (2.8,0.7) to node[below]{$\rho_i$} (3.2,0.7);
\node[] at (3,1.4){$\tau_{s(\alpha_i)}$};
\end{tikzpicture}
\end{center}

Suppose $\tau_{s(\alpha_1^{\epsilon_1})} = \tau_{s(w)}$ is not a pending arc. Then, $\tau_{s(w)}$ borders two triangles. There already exists a segment $\rho_1$ in one of these triangles. Let the other triangle be denoted $\Delta_0$ and let $v_0$ be the vertex of $\Delta_0$ which is not incident to $\tau_{s(\alpha_1^{\epsilon_1})}$.  In this case, we set $\rho_0$ to be the arc between $v_0$ and $\tau_{s(\alpha_1^{\epsilon_1})}$,  oriented from $v_0$ to $\tau_{s(\alpha_1^{\epsilon_1})}$. If $\tau_{t(\alpha_m^{\epsilon_m})} = \tau_{t(w)}$ is not a pending arc, we set $\rho_{m+1}$ similarly. 

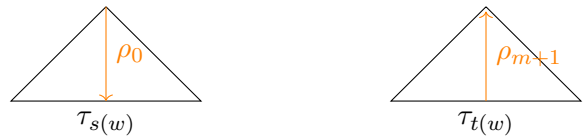
\begin{center}
\begin{tikzpicture}[scale=1.25]
     \node[above,  xshift = -5] at (0.5,0.5){};
     \draw(1,1) to (0,0) to node[below]{$\tau_{s(w)}$}(2,0) to (1,1);
     \node[above, xshift = 5] at (1.5,0.5){};
     \draw[->, orange] (1,1) to node[right]{$\rho_0$} (1,0);    
   \node[above,  xshift = -5] at (4.5,0.5){};
     \draw(5,1) to (4,0) to node[below]{$\tau_{t(w)}$}(6,0) to (5,1);
     \node[above, xshift = 5] at (5.5,0.5){};
     \draw[<-, orange] (5,0.95) to node[right]{$\rho_{m+1}$} (5,0);    
 \end{tikzpicture}
 \end{center}

Next, suppose $\tau_{s(\alpha_1^{\epsilon_1})}$ is a pending arc,  and $\tau_{s(\alpha_1^{\epsilon_1})} = \tau_{t(\alpha_1^{\epsilon_1})}$,  so that $\alpha_1$ is a loop.  Let $\Delta_0$ be the triangle which this pending arc borders.  If $\Delta_0$ has two distinct vertices, let $v_0$ be the marked point in this bigon which is not incident to  $\tau_{s(\alpha_1^{\epsilon_1})}$,  and let $\rho_0$ be the arc between $v_0$ and $\tau_{s(w)}$,  oriented from $v_0$ to $\tau_{s(w)}$. If $\Delta_0$ only has one vertex, then $\Delta_0$ consists of $\tau_{s(w)}$, another pending arc $\sigma_p$, and a standard arc $\sigma_s$. Note that we can never see a triangle consisting of three pending arcs since this would imply that the endpoint of the pending arcs is a puncture. Let $v_0$ be the unique endpoint of $\Delta_0$. Suppose that $\tau_{s(w)}$ follows $\sigma_p$ in counterclockwise order. Then, we let $\rho_0$ be the arc between $v_0$ and $\tau_{s(w)}$, oriented from $v_0$ to $\tau_{s(w)}$ which follows $\sigma_s$ in clockwise direction. 
If $\alpha_m$ is a loop, we do something similar.

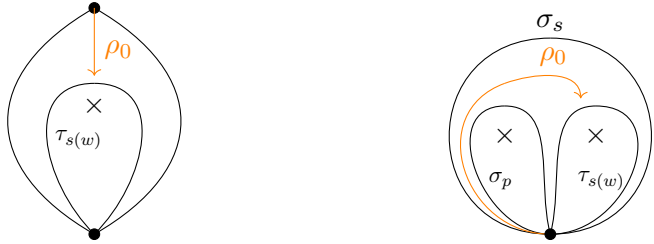
\begin{center}
\begin{tikzpicture}[scale = 2]
\draw[out=30,in=-30,looseness=1.5] (0,0) to (0,1.5);
\draw[out=150,in=-150,looseness=1.5] (0,0) to (0,1.5);
\draw[in=0,out=45,looseness =1.25] (0,0) to (0,1);
\draw[in=180,out=135,looseness=1.25] (0,0) to node[right, scale = 0.8] {$\tau_{s(w)}$} (0,1);
\draw[thick] (0,0.85) node {$\mathbf{\times}$};
\draw[fill=black] (0,0) circle [radius=1pt];
\draw[fill=black] (0,1.5) circle [radius=1pt];
\draw[->, orange](0,1.5) to node[right, yshift = -3pt]{$\rho_0$} (0,1.05);
\draw[out=75,in=-180] (3,0) to (3.3,0.85);
\draw[out=0,in=0, looseness = 1.5] (3,0) to node[midway,left,yshift=-4pt, scale = 0.8]{$\tau_{s(w)}$} (3.3,0.85);
\draw[out=105,in=0] (3,0) to (2.7,0.85);
\draw[out=180,in=180, looseness = 1.3] (3,0) to node[midway,right,yshift=-4pt, xshift = 2pt, scale = 0.8]{$\sigma_p$} (2.7,0.85);
\draw[looseness=1.75,out=0,in=0] (3,0) to  (3,1.3);
\draw[looseness=1.75,out=180,in=180] (3,1.3) to (3,0);
\draw[thick] (2.7,0.65) node {$\mathbf{\times}$};
\draw[thick] (3.3,0.65) node {$\mathbf{\times}$};
\node [] at (3,1.4) {$\sigma_s$};
\draw[out = 180, in = -160, looseness = 1.4, orange] (3,0) to (2.7,1); 
\draw [out = 20, in = 90, ->, orange] (2.7,1) to node[above]{$\rho_0$} (3.2,0.9);
\draw[fill=black] (3,0) circle [radius=1pt];
\end{tikzpicture}\\
\end{center}

The final case is when $\tau_{s(\alpha_1^{\epsilon_1})}$ is a pending arc,  but $\tau_{s(\alpha_1^{\epsilon_1})} \neq \tau_{t(\alpha_1^{\epsilon_1})}$.  In this case, we draw $\rho_0$ as an arc which starts at the unique marked point incident to $\tau_{s(w)}$ and winds once around the orbifold point, creating a self-intersection.  We choose as a convention winding in the positive direction (counterclockwise); however,  since the orbifold point is order 3, this is isotopic to winding once in the negative direction (clockwise).  If  $\tau_{t(\alpha_m^{\epsilon_m})}$ is a pending arc,  but $\tau_{s(\alpha_m^{\epsilon_m})} \neq \tau_{t(\alpha_m^{\epsilon_m})}$, we do something similar. 

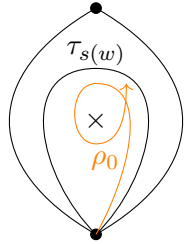
\begin{center}
\begin{tikzpicture}[scale = 2]
\draw[out=30,in=-30,looseness=1.5] (0,0) to (0,1.5);
\draw[out=150,in=-150,looseness=1.5] (0,0) to (0,1.5);
\draw[in=0,out=45,looseness =1.25] (0,0) to (0,1.1);
\draw[in=180,out=135,looseness=1.25] (0,0) to (0,1.1);
\draw[thick] (0,0.75) node {$\mathbf{\times}$};
\draw[fill=black] (0,0) circle [radius=1pt];
\draw[fill=black] (0,1.5) circle [radius=1pt];
\draw[orange, out = 60, in = 0] (0,0) to node[left, yshift = -8pt]{$\rho_0$} (0,1);
\draw[orange, out = 180, in = 180, looseness = 1.2](0,1) to (0,0.6);
\draw[orange, out = 0, in = 270, ->] (0,0.6) to (0.2,1);
\node[] at (0,1.2){$\tau_{s(w)}$};
\end{tikzpicture}
\end{center}

Since our string satisfies $t(\alpha_i^{\epsilon_i}) = s(\alpha_{i+1}^{\epsilon_{i+1}})$,  we can connect all $\rho_i$ to create an arc $\gamma(w)$ whose corresponding string is $w$.  Up to isotopy, there is a unique way to do this. 

We can follow the same line of reasoning for a band $\overline{w} = \alpha_1^{\epsilon_1}\cdots \alpha_m^{\epsilon_m}$, making curves $\rho_1,\ldots,\rho_m$ and connecting these to form a closed curve $\xi_{\overline{w}}$.
\end{proof}

\begin{remark}
Baur and Coelho Sim{\~o}es give a geometric model for any gentle algebra via a dissection of a surface in \cite{baur2021geometric}; a similar model for the derived category of a gentle algebra appears in \cite{opper2018geometric}. The former shows that indecomposable modules of the algebra correspond to equivalence classes of \emph{permissible curves} on this surface. Given an orbifold $\mathcal{O}$ with triangulation $T$ and corresponding gentle algebra $KQ_T/I_T$, we can go from the pair $\mathcal{O}, T$ to the surface with dissection from \cite{baur2021geometric} by replacing all pending arcs and their interior with monogons which each enclose an unmarked boundary component. In an orbifold, arcs which wind too many times around an orbifold point are isotopic to those with less winding; in the surface with dissection model, this property is replaced with a condition that makes arcs with too high of a winding number not permissible.
\end{remark}

Combining Theorems \ref{thm:BRString} and \ref{thm:StringsBijectCurves} shows us that the indecomposable objects in $\text{rep}(KQ_T/I_T)$ correspond to arcs and closed curves in $\mathcal{O}$.

\begin{corollary}\label{cor:ArcsGiveIndecomp}
Given an orbifold $\mathcal{O}$ with triangulation $T$, the indecomposable objects in $\text{rep}(Q_T,I_T)$ are in bijection with (possibly generalized) arcs and closed curves in $\mathcal{O}$.
\end{corollary}

\subsection{Morphisms and AR Translation} 

We now know that the indecomposable objects in $\text{rep}(Q_T,I_T)$ are in bijection with arcs and closed curves on $\mathcal{O}$. Next, we compare combinatorial descriptions of irreducible morphisms with the geometry on an orbifold.  Irreducible morphisms were described in terms of string combinatorics in \cite{butler1987auslander}.

\begin{lemma}[\cite{butler1987auslander}]\label{lem:BRIrredMorphisms}

Let $KQ/I$ be a gentle algebra.

\begin{enumerate}
    \item Let $w$ be a string which does not start (end) on a peak. Then, the canonical embedding $M(w) \to M({}_h w)$ ($M(w) \to M(w_h)$) is irreducible.
    \item Let $w$ be a string which starts (ends) on a peak, and suppose that $w$ is not a direct (inverse) string. Then, the canonical projection $M(w) \to M({}_{c^{-1}}w)$ ($M(w) \to M(w_{c^{-1}})$) is irreducible.
\end{enumerate}

Moreover, this is a complete list of all irreducible morphisms in $\text{rep}(Q_T,I_T)$.
\end{lemma}

Br{\"u}stle and Zhang show that adding a hook or removing a cohook from a string $w$ corresponds to moving one endpoint of $\gamma(w)$ \cite{brustle2011cluster}. We show that the same holds true in an orbifold. 

Given a marked point $b$, let $b^+$ be the marked point on the same boundary component as $b$ which is immediately counterclockwise from $b$; it is possible that $b = b^+$. Define $b^-$ as the marked point which is immediately clockwise from $b$. Given an arc $\gamma$ with an orientation, let ${}^+\gamma$ denote the arc which is isotopic to $\gamma$ except we  adjust $\gamma$ to start at $s(\gamma)^+$. To be precise, let $\sigma$ denote the portion of the boundary component between $s(\gamma)^+$ and $s(\gamma)$, oriented from $s(\gamma)^+$ to $s(\gamma)$. Then, ${}^+\gamma$ is the result of composing $\sigma$ and $\gamma$ with their given orientations and avoiding $s(\gamma)$.  Let $\gamma^+$ be the arc which is isotopic to $\gamma$ except it ends at $t(\gamma)^+$. Define ${}^-\gamma$ and $\gamma^-$ analogously. 

\begin{lemma}\label{lem:RotateStringArcOnce}
Let $T$ be a triangulation of an orbifold $\mathcal{O}$ with gentle pair $(Q_T,I_T)$. Let $w$ be a string and $\gamma = \gamma(w)$. 
\begin{enumerate}
\item Suppose $w$ does not start (end) on a peak (deep). Then $\gamma({}_h w) = {}^+\gamma$ ($\gamma(w_h) = \gamma^+$).
\item Suppose $w$ starts (ends) on a peak and that $w$ is not a direct (inverse)  string. Then, $\gamma({}_{c^{-1}}w) =  {}^+\gamma$ ($\gamma(w_{c^{-1}}) = \gamma^+$).
\end{enumerate}
Moreover, there exist parallel conditions describing ${}^-\gamma$ and $\gamma^-$, swapping the roles of hooks and cohooks. 
\end{lemma}

\begin{proof}
Whether $\gamma = \gamma(w)$ starts on a peak is determined by the configuration around $\gamma$ before its first crossing with $T$. Let $\tau_1$ be the first arc which $\gamma$ crosses. If $\gamma$ passes through a triangle $\Delta_0$ before its first crossing with $T$, then $w$ starts on a peak if and only if the clockwise neighbor of $\tau_1$ in $\Delta_0$ is a boundary edge. If this arc is a boundary arc, then ${}^+\gamma$ no longer crosses $\tau_1$, nor $\tau_2,\ldots,\tau_k$ where $\tau_1,\ldots,\tau_k$ form a fan  (a set of arcs which pairwise border the same triangle and all share an endpoint) based at ${}^+ s(\gamma)$. This is true even if some of the arcs are pending arcs. Thus, we see that the arc $\gamma({}_{c^{-1}}w)$ associated to the string ${}_{c^{-1}}w$ is equivalent to the arc ${}^+\gamma$ up to isotopy. 

If the clockwise neighbor, $\tau_a$, of $\tau_1$ in $\Delta_0$ is not a boundary edge, then ${}^+ \gamma$ will cross this arc as well as all arcs which are based at ${}^+ s(\gamma)$ and follow $\tau_a$ in clockwise direction. Thus in this case, we see that ${}^+\gamma$ is isotopic to the arc $\gamma({}_hw)$ associated to ${}_hw$. 

Finally, suppose $\gamma$ winds around an orbifold point before its first intersection with $T$. The corresponding string $w_\gamma = \alpha_1^{\epsilon_1}\cdots \alpha_m^{\epsilon_m}$ has $s(w_\gamma)$ a vertex in $Q_T$ which is incident to a loop $\delta$, but $\alpha_1 \neq \delta$. Therefore, $w_\gamma$ will never start on a peak in this case since $\delta w_\gamma$ is a valid string. Similar to the previous discussion, ${}^+ \gamma$ will now cross the pending arc $\rho$ enclosing this orbifold point twice as well as the fan of edges which are clockwise of $\rho$. We see again that ${}^+\gamma$ and $ \gamma({}_h w)$ are equivalent.

The other three cases can be shown in a parallel manner. 
\end{proof}

Combining Lemmas \ref{lem:BRIrredMorphisms} and \ref{lem:RotateStringArcOnce} immediately gives a description of irreducible morphisms in a gentle algebra from an orbifold. 

\begin{proposition}\label{prop:IrredMorphOrbifold}
Let $T$ be a triangulation of an orbifold $\mathcal{O}$ with gentle pair $(Q_T,I_T)$. Let $w$ be a string and $\gamma = \gamma(w)$. Then, all irreducible morphisms from $M(\gamma)$ are of the form $M(\gamma) \to M(\gamma^+)$ and $M(\gamma) \to M({}^+ \gamma)$. 
\end{proposition}

Butler and Ringel also provide a combinatorial description of Auslander-Reiten translation for string modules in \cite{butler1987auslander}. Let $w$ be a string for a gentle pair such that $M(w)$ is not projective.  Then, $\tau M(w) = M(w')$ where $w'$ depends on whether $w$ starts or ends on a deep. If $w$ starts on a deep, we remove a hook from the beginning of $w$; otherwise, we add a cohook to the beginning. If we then repeat this at the end of $w$, we reach $w'$; the fact that $M(w)$ is not projective guarantees that this process is well-defined. There is a parallel procedure for finding the string associated to $\tau^{-1}M(w)$.

 This result along with Lemma \ref{lem:RotateStringArcOnce} yields a description of how Auslander-Reiten translation acts on the level of arcs on $\mathcal{O}$. 

\begin{theorem}\label{thm:ARTranslation}
Let $T$ be a triangulation of an orbifold $\mathcal{O}$ with gentle pair $(Q_T,I_T)$. Let $w$ be a string and $\gamma = \gamma(w)$. Then, if $M(\gamma)$ is not projective, $\tau M(\gamma) = M({}^- \gamma^-)$ and if $M(\gamma)$ is not injective, $\tau^{-1} M(\gamma) = M({}^+ \gamma^+)$.
\end{theorem}

A string module $M(w)$ is projective only if there is a unique vertex at the top of $w$ (see Definition \ref{def:TopAndBottom}) and $w$ starts and ends on a deep; the conditions for an injective string module are parallel. From the description of an arc $\gamma(w)$ from a string $w$, we can see that $M(\gamma)$ is projective if $\gamma = {}^+\tau^+$ for some $\tau \in T$.

When an orbifold has no orbifold points, so that it is in fact a surface, versions of Proposition \ref{prop:IrredMorphOrbifold} and Theorem \ref{thm:ARTranslation} in the context of the cluster category are given in \cite{brustle2011cluster}.

\section{Caldero-Chapoton Map}\label{sec:CCMap}

Caldero and Chapoton described a map from the category of representations of $Q$ to the cluster algebra corresponding to $Q$ when $Q$ is an ADE type quiver \cite{CCMap}. This map was generalized to apply to modules over any basic algebra in \cite{cerulli2015caldero}. In this section, we compare the application of this map to a string module $M(w_\gamma)$ or band module $M(\overline{w}_\xi)$ with the snake graph expansion formula applied to an arc $\gamma$ or closed curve $\xi$ respectively. We remark that a similar result in the case of a cluster algebra from a surface is given in Section 11 of \cite{geiss2022schemes}.

\subsection{Quiver Grassmannians}

Given a basic algebra $\Lambda = KQ/I$, let $M \in \text{rep}(Q,I)$. Recall the dimension vector of $M$, $\dim(M)$ is given by $\mathbf{d} = (d_1,\ldots,d_n)$ where $d_i = \dim(M_i)$ and $n = \vert Q_0 \vert$. For $\mathbf{e} \in \mathbb{N}^{n}$, we define $\text{Gr}_{\mathbf{e}}(M)$ to be the variety of subrepresentations of $M$ with dimension vector $\mathbf{e}$. Let $\chi$ be the Euler-Poincar\'e characteristic. 

By Theorem 1.2.a and Corollary 3.1 from \cite{haupt2012euler}, given a string or band module $M$, we can compute $\chi (\text{Gr}_{\mathbf{e}}(M))$ by counting the canonical string submodules of $M$ with this dimension vector. When $M$ is a string module, {\c{C}}anak{\c{c}}i and Schroll show we can equivalently compute $\chi(\text{Gr}_{\mathbf{e}}(M))$ by counting certain perfect matchings of a corresponding snake graph. They do this by comparing the lattice of canoncial submodules of $M(w)$ and the lattice of perfect matching of a snake graph $\mathcal{G}(w)$ associated to $w$. We review the definition of $\mathcal{G}(w)$. 

If $w = e_v$, $\mathcal{G}(w)$ will have 1 tile, $G_1$, with label $v$. Now, let $w = \alpha_1^{\epsilon_1} \cdots \alpha_m^{\epsilon_m}$ be a string for $m \geq 1$. The snake graph $\mathcal{G}(w)$ will have $m+1$ tiles $G_1,\ldots,G_{m+1}$, where $G_i$ will have label $w(i)$. If $\epsilon_1 = +$, we glue $G_2$ on the north edge of $G_1$; otherwise, $\epsilon_1 = -$ and we glue $G_2$ on the east edge of $G_1$. Then, for all $2 \leq i \leq m$, we glue $G_{i+1}$ onto either the north or the east edge of $G_i$ so that $G_{i-1}, G_i$ and $G_{i+1}$ share a vertex if and only if $\epsilon_{i-1} = \epsilon_i$. We do not include edge labels to $\mathcal{G}(w)$. This is the inverse of the description in Section \ref{sec:LatticePMs} for the partial order $P$ associated to a snake graph $\mathcal{G}$ such that the lattice of order ideals of $P$ is isomorphic to the lattice of perfect matchings of $\mathcal{G}$. 
For example, we give the snake graph $\mathcal{G}(w)$ associated to a string of the form $w = \alpha_1 \alpha_2 \alpha_3^{-1} \alpha_4^{-1} \alpha_5$. The tile labels would follow from the labels of the vertices of $w(1),\ldots,w(6)$.

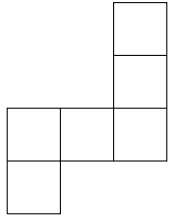
\begin{center}
\begin{tikzpicture}[scale = 0.7]
\draw (0,0) -- (1,0) -- (1,1) -- (3,1) -- (3,4) -- (2,4) -- (2,2) -- (0,2) -- (0,0);
\draw (0,1) -- (1,1) -- (1,2);
\draw (2,1) -- (2,2) -- (3,2);
\draw (2,3) -- (3,3);
\end{tikzpicture}
\end{center}


\begin{theorem}[Theorem 3.18 \cite{ccanakcci2021lattice}]\label{thm:CaScLatticeBij}
Given a string $w$ with snake graph $\mathcal{G}(w)$, there is a lattice bijection between the lattice of perfect matchings of $\mathcal{G}(w)$ and the canonical submodule lattice of $M(w)$. 
\end{theorem}

We remark that this bijection will send a perfect matching $P$ with height vector $\mathbf{h}(P)$ to a submodule $M$ with $\dim(M) = \mathbf{h}(P)$. 

\begin{corollary}[Corollary 3.10 \cite{ccanakcci2021lattice}, Remark 5.5 \cite{MSW-bases}]\label{cor:CaScEulerChar}
Let $w$ be a string for a gentle pair $(Q,I)$. Then, $\chi(\text{Gr}_{\mathbf{e}}(M(w)))$ is equal to the number of perfect matchings of $\mathcal{G}(w)$ with height vector $\mathbf{e}$. 
\end{corollary}

We provide an analogous result for band modules of the form $M(\overline{w},\lambda,1)$. This was also given as Remark 5.8 in \cite{MSW-bases}, but we work abstractly without associating the band to a closed curve. Thus, our result can be applied to algebras other than those associated to triangulated surfaces and orbifolds.  

First, we describe the band graph $\widetilde{\mathcal{G}}(\overline{w})$ associated to a band $\overline{w} = \alpha_1^{\epsilon_1} \cdots \alpha_m^{\epsilon_m}$. Given a tile $G_i$ of a snake graph, let $N(G_i)$ denote the north edge of $G_i$ and similarly for $E(G_i),S(G_i),$ and $W(G_i)$. Let $\mathcal{G} = \mathcal{G}(\alpha_1^{\epsilon_1} \cdots \alpha_{m-1}^{\epsilon_m})$ have tiles $G_1,\ldots,G_m$. We form  $\widetilde{\mathcal{G}}(\overline{w})$ by taking $\mathcal{G}$ and gluing $G_1$ and $G_m$, where the specific edges glued depend on the parity of $m$ and the sign of $\epsilon_m$:
\begin{enumerate}
    \item If $\epsilon_m = +$ and $m$ is even, glue $W(G_1)$ to $E(G_m)$.
    \item If $\epsilon_m = +$ and $m$ is odd, glue $W(G_1)$ to $N(G_m)$.
    \item If $\epsilon_m = -$ and $m$ is even, glue $S(G_1)$ to $N(G_m)$.
    \item If $\epsilon_m = -$ and $m$ is odd, glue $S(G_1)$ to $E(G_m)$.
\end{enumerate}

\begin{theorem}\label{thm:BandMatchingSubmoduleBijection}
Given a band $\overline{w}$ from a gentle pair $(Q,I)$ with associated band graph $\widetilde{\mathcal{G}}(\overline{w})$, there is a lattice bijection between the lattice of good matchings of $\widetilde{\mathcal{G}}(\overline{w})$ and the canonical submodule lattice of $ M(\overline{w}, \lambda,1)$ for $\lambda \in K^\times$. 
\end{theorem}

\begin{proof}
Given a band $\overline{w} = \alpha_1^{\epsilon_1}\cdots \alpha_m^{\epsilon_m}$, let $w'$ be the string $\alpha_1^{\epsilon_1}\cdots \alpha_{m-1}^{\epsilon_{m-1}}$. Since we assume $\overline{w}$ came from a gentle pair, $\overline{w}$ cannot have all direct or inverse arrows. Therefore, we may assume without of loss of generality that $\epsilon_m = -$. For convenience, let $m$ be even; the odd $m$ case is analogous. 

Let $\widetilde{\mathcal{G}} = \widetilde{\mathcal{G}}(\overline{w})$ and $\mathcal{G} = \mathcal{G}(w')$. From Theorem \ref{thm:CaScLatticeBij}, we know there is a lattice bijection between the lattice of perfect matchings of $\mathcal{G}$ and the lattice of canonical submodules of $M(w')$.  Recall some perfect matchings of $\mathcal{G}(w')$ descend to good matchings of $\mathcal{G}(w)$; these are those which use at least one of the edges in $\mathcal{G}$ which are glued to form $\widetilde{\mathcal{G}}$. Similarly, some of the submodules of $M(w')$ can also be viewed as submodules of $M(\overline{w})$. Since we assume $\epsilon_m = -$, these submodules are those such that, if the basis element $z_{1} \in M(w')_{s(\alpha_m)}$ associated to the first vertex of $w$ is included in the submodule, so is $z_{m} \in M(w')_{t(\alpha_m)}$, the basis element associated to the last vertex of $w$.

Our method of proof will be to show that the bijection of \canakci and Schroll restricts to a bijection between the perfect matchings of $\mathcal{G}$ which do not lift to good matchings and the submodules of $M(w')$ which are not also submodules of $M(\overline{w})$. We call these \emph{bad matchings} and \emph{bad submodules} respectively. We will see that each of these sets form an interval in the corresponding lattice. The consequence will be that we also have a bijection between the lattice of good matchings of $\widetilde{\mathcal{G}}$ and the lattice of submodules of $M(\overline{w})$. 

Let $z_1,\ldots,z_m$ be the canonical basis elements of $M(w')$. The set of bad submodules of $M(w')$ consists of all submodules containing $z_1$ and not containing $z_m$. First, assume $w'$ is not a direct string. Let $c_1$ be the first deep in $w'$ and let $d_\ell$ be the last peak in $w'$; it is possible that $c_1 = 1$ and $d_\ell = m$. Assume first $c_1+1 < d_\ell$.  Then, the set of bad submodules forms an interval in the submodule lattice and is in bijection with the set of canonical submodules of the string $\alpha_{c_1+1}^{\epsilon_{c_1}+1} \cdots \alpha_{d_\ell-2}^{\epsilon_{d_\ell-2}}$, where if $c_1+1 = d_\ell-1$, this is the length 0 string $e_{s(\alpha_{c_1})} = e_{t(\alpha_{d_{\ell-1}})}$. This is true because, in order to have $z_1$ as a basis element in a submodule, we must also include $z_2,\ldots,z_{c_1}$ as basis elements, and if $z_m$ is not a basis element, we cannot have any of the elements $z_{d_\ell},\ldots,z_{m-1}$.

We compare this with the set of perfect matchings of $\mathcal{G}$ which use neither $S(G_1)$ nor $N(G_m)$. Since $S(G_1)$ is in the minimal matching by assumption, in order to reach any bad matching from the minimal matching, we must twist $G_1$. Either $G_2$ is glued to the north edge of $G_1$, which means $G_1$ is twistable in the minimal matching and $c_1 = 1$,  or $ c_1 > 1$ and by Lemma \ref{lem:ZigZagFlip}  in order to twist $G_1$ we must first twist $G_{c_1} G_{c_1-1},\ldots,G_2,G_1$ in this order. From Lemma \ref{lem:twistparity} and our assumption that $m$ is even, we know that $N(G_m)$ is in the minimal matching of $\mathcal{G}$, so in order to have a bad matching we cannot twist $G_m$. Using the same logic as Lemma \ref{lem:ZigZagFlip}, this means we also cannot twist tiles $G_{d_\ell}, G_{d_\ell+1},\ldots,G_{m-1}$. Note that, since $w'$ is not a direct or inverse string, $\mathcal{G}$ will not be a zig-zag shape, so these conditions will not interfere with each other. Such a matching will be a bad matching regardless of which edges it includes from the subgraph on tiles $G_{c_1+1},\ldots,G_{d_\ell-1}$, so the bad matchings form an interval which is isomorphic to the lattice of all perfect matchings of this subgraph.  By Theorem \ref{thm:CaScLatticeBij}, this interval is isomorphic to the interval of bad submodules described previously.

Now, we address the special cases. First, suppose $w'$ is not a direct string and $c_1+1 = d_\ell$. Then, there is exactly one bad submodule, which is the span of $z_1,\ldots,z_{c_1}$. Using reasoning as in the general case, there will be one bad matching, which is the result of twisting $G_{c_1},\ldots,G_1$ and not twisting $G_m$, which means we cannot twist $G_{m-1},\ldots,G_{d_\ell} = G_{c_1+1}$. Since the height vector of this unique bad matching is equal to the dimension vector of the bad submodule, we again see that they correspond by the bijection in Theorem \ref{thm:CaScLatticeBij}.

Finally, assume $w'$ is a direct string. Then there are no bad submodules of $M(w')$ because a submodule including $z_1$ must also include $z_m$. In this case, $\mathcal{G}$ is a zig-zag snake graph with an even number of tiles and with $G_2$ glued onto $N(G_1)$. One can show with induction that $m$ of the $m+1$ perfect matchings of $\mathcal{G}$ use $S(G_1)$ and the other matching uses $N(G_m)$. 

\end{proof}

We can apply results from \cite{haupt2012euler} to conclude that we can compute $\chi(Gr_{\mathbf{e}}(M(\overline{w})))$ with good matchings of band graphs. 

\begin{corollary}\label{cor:EulerCharBandGraph}
Let $\overline{w}$ be a band for a gentle pair $(Q,I)$. Then, $\chi(Gr_{\mathbf{e}}(M(\overline{w},\lambda,1)))$ is equal to the number of good matchings of $\widetilde{\mathcal{G}}(\overline{w})$ with height vector $\mathbf{e}$. 
\end{corollary}

We remark that, when our algebra $KQ/I$ comes from an orbifold with triangulation $T$, so that a string $w$ corresponds to an arc $\gamma(w)$, then $\mathcal{G}(w)$ is isomorphic to $\mathcal{G}_{\gamma(w),T}$, and similarly $\widetilde{\mathcal{G}}(\overline{w})$ is isomorphic to $\widetilde{\mathcal{G}}_{\xi_{\overline{w}},T}$ for a band $\overline{w}$ with corresponding closed curve $\xi_{\overline{w}}$. 

\subsection{Minimal Terms}\label{subsec:gvector}

First we define the $g$-vector of a representation of a bound quiver. 
We follow the definition in \cite{cerulli2015caldero} but we restrict to the case where our representations never have decoration. For $i \in Q_0$, let $S_i$ be the simple representation supported at vertex $i$.

\begin{definition}
Let $\Lambda = KQ/I$ and let $n = \vert Q_0 \vert$. Let $M \in \text{rep}(Q,I)$. Then, $\mathbf{g}_\Lambda(M) = (g_1(M),\ldots,g_n(M))$ where \[
g_i(M) = -\dim \Hom_{\Lambda} (S_i,M) + \dim \Ext_\Lambda^1(S_i,M)
\]
\end{definition}

We can compute the $g$-vector for a representation by looking at its injective presentation.

\begin{lemma}[Lemma 3.4 \cite{labardini2019family}]\label{lem:gvectorFromInjectivePres}
Let $M \in \text{rep}(Q,I)$ with $g$-vector $\mathbf{g}_{\Lambda}(M) = (g_1(M),\ldots,g_n(M))$. Consider a minimal injective presentation of $M$ given by \[
0 \to M \to \oplus_{i=1}^n I_i^{a_i'} \to \oplus_{i=1}^n I_i^{b_i'}.
\]
Then, \[
g_i(M) = -a_i' + b_i'
\]
\end{lemma}

Since we will be working with gentle algebras, we can describe the minimal injective presentation, and thus the $g$-vector, for an indecomposable representation based on the combinatorics of its associated string or band. 

A dual version of the following is given in \cite{palu1707non} using projective presentations. Recall we identify $Q_0$ with $[n] = \{1,\ldots,n\}$.

\begin{proposition}[Proposition 1.43 and Corollary 1.44 \cite{palu1707non}]\label{prop:gvectorfromstring}
Let $\Lambda = KQ/I$ for a gentle pair $(Q,I)$, and let $M = M(w) \in \text{rep}(Q,I)$ be a string module. 
\begin{itemize}
    \item Let $\mathbf{a}_\Lambda(M) = (a_1,\ldots,a_n)$ where $a_i$ is the number of times $i \in Q_0$ is a deep in $w$.
    \item Let $\mathbf{b}_\Lambda(M) = (b_1,\ldots,b_n)$ where $b_i$ is the number of times $i \in Q_0$ is a strict peak in $w$.
    \item Let $\mathbf{r}_\Lambda(M) = (r_1,\ldots,r_n)$ be defined by \[
    \mathbf{r}_\Lambda(M) = 
    \begin{cases} 
    0 & w \text{ starts and ends in a peak} \\
    e_{s(\alpha)} & w \text{ ends in a peak and } \alpha w \text{ is a string with } \alpha \in Q_1\\
    e_{s(\beta)}  & w \text{ starts in a peak and } w \beta^{-1} \text{ is a string with } \beta \in Q_1\\
    e_{s(\alpha)} + e_{s(\beta)} & \alpha w \text{ and  } w\beta^{-1} \text{ are strings with } \alpha, \beta \in Q_1\\
    \end{cases}
    \]
    where $e_i \in \mathbb{Z}^n$ is the standard basis vector with a 1 in position $i$. 
    
\end{itemize}  
Then, the injective presentation of $M$ is given by \[
0 \to M \to \bigoplus_{i=1}^n I_i^{\oplus a_i} \to \bigoplus_{i=1}^n I_i^{\oplus (b_i+r_i)}
\]
which by Lemma \ref{lem:gvectorFromInjectivePres}  implies that the $g$-vector of $M$ is given by \[
\mathbf{g}_\Lambda(M) = -\mathbf{a}_\Lambda(M) + \mathbf{b}_\Lambda(M) + \mathbf{r}_\Lambda(M).
\]
\end{proposition}

We provide a version of Proposition \ref{prop:gvectorfromstring} for band modules. 

\begin{proposition}\label{prop:gvectorfromband}
Let $\Lambda = KQ/I$ for a gentle pair $(Q,I)$, and let $M = M(\overline{w},\lambda,1) \in \text{rep}(Q,I)$ be a band module.
\begin{itemize}
    \item Let $\mathbf{a}_\Lambda(M) = (a_1,\ldots,a_n)$ where $a_i$ is the number of times $i \in Q_0$ is a deep in $\overline{w}$.
    \item Let $\mathbf{b}_\Lambda(M) = (b_1,\ldots,b_n)$ where $b_i$ is the number of times $i \in Q_0$ is a peak in $\overline{w}$.
\end{itemize}    

Then, the injective presentation of $M$ is given by \[
0 \to M \to \bigoplus_{i=1}^n I_i^{\oplus a_i} \to \bigoplus_{i=1}^n I_i^{\oplus b_i}
\]
and the $g$-vector is given by \[
\mathbf{g}_\Lambda(M) = -\textbf{a}_\Lambda(M) + \textbf{b}_\Lambda(M)
\]

\end{proposition}

\begin{proof}

For simplicity, we choose a cyclic representative of $\overline{w}$ such that $\epsilon_1 = +$ and $\epsilon_m = -$. This is always possible since the algebra $KQ/I$ is finite-dimensional, so we cannot have any bands consisting of only direct or inverse arrows. 

Let $z_1,\ldots,z_m$ be the canonical basis elements of $M$. Let  $c_1,\ldots,c_\ell \in [m]$ be the complete list for indices such that each $\overline{w}(c_i)$ is a deep, and similarly let $d_1,\ldots,d_{\ell}$ be the complete list such that each $\overline{w}(d_i)$ is a  peak. By our assumption, $d_1 = 1$.

The definition of a band module implies that the socle of $M$ is generated by $ z_{c_1},\ldots,z_{c_\ell}$. The injective envelope is therefore of the form $f: M \to K = \bigoplus_{i=1}^\ell I_{\overline{w}(c_i)}$. Recall that the injective module $I_{x}$ corresponds to the string which starts and ends on a peak and has $x$ as its unique deep. For each $I_{\overline{w}(c_i)}$, we choose the representative of the associated string which matches the orientation of the arrows in  $\overline{w}$ near $c_i$. Using this orientation, index the canonical basis elements of $I_{\overline{w}(c_i)}$ as $z^{(i)}_{c_i - a},\ldots,z_{c_i}^{(i)},\ldots,z^{(i)}_{c_i+b}$ for some $a,b \geq 0$.

With this indexing, the map $f$ is determined by the images  $f(z_{c_i}) = z_{c_i}^{(i)}$ as we will now describe.  Suppose that $j \in [m]$ is such that $\overline{w}(j)$ is neither a deep nor a peak; then, $j$ is connected to a unique deep, say $c_i$, by a direct or inverse substring of $\overline{w}$. It follows that $f(z_j) = z^{(i)}_j$. If $j = d_i$ for $i > 1$, then $j$ is connected by an inverse substring to a deep $c_{i-1}$ and by a direct substring to a deep $c_i$, and we have $f(z_{d_i}) = z_{d_i}^{(i-1)} + z_{d_i}^{(i)}$. Similarly, the peak $d_1$ is connected by a direct substring to the deep $c_1$ and by an inverse substring to $c_\ell$, and the definition of this band module implies that the image is $f(z_1) = z_1^{(1)} + \lambda z_{m+1}^{(\ell)}$. Note that the basis element $z_{m+1}^{(\ell)}$ exists in $I_{w(c_\ell)}$ by our choice of indexing.

Changing the basis at each vector space $M_{w(d_i)}$, we see that the socle of $\coker(f)$ is generated by $ z_1^{(1)} - \lambda z_{m+1}^{(\ell)}, z_{d_2}^{(1)} - z_{d_2}^{(2)},\ldots,z_{d_\ell}^{(\ell-1)} - z_{d_\ell}^{(\ell)}$. Therefore, the injective envelope of $\coker(f)$ is given by $ \bigoplus_{i=1}^\ell I_{\overline{w}(d_i)}$. The explicit details of the map $K \to \bigoplus_{i=1}^\ell I_{\overline{w}(d_i)}$ are similar to the map $M \to K$.

\end{proof}

Our next goal is to compare the $g$-vector of $M(w)$ or $M(\overline{w})$ and the weight of the minimal matching of $\mathcal{G}_{\gamma(w),T}$ or $\widetilde{\mathcal{G}}_{\xi(\overline{w}),T}$ respectively. 
We begin with a technical lemma. While this fact is well-known, we provide a line of reasoning for any reader unfamiliar with snake graphs.

\begin{lemma}\label{lem:ZigZagFlip}
Let $\mathcal{G} = (G_1,\ldots,G_m)$ be a zig-zag snake graph. 
\begin{itemize}
    \item If $G_2$ is glued on the east edge of $G_1$, then in the minimal matching of $\mathcal{G}$, $G_m$ is the unique twistable tile. In the maximal matching of $\mathcal{G}$, $G_1$ is the unique twistable tile.
    \item If $G_2$ is glued on the north edge of $G_1$, then in the minimal matching of $\mathcal{G}$, $G_1$ is the unique twistable tile. In the maximal matching of $\mathcal{G}$, $G_m$ is the unique twistable tile.
\end{itemize}
\end{lemma}

\begin{proof}
The proof follows from two basic observations. First, note that if $G = (G_1,\ldots,G_m)$ is a zig-zag shape, each tile has a vertex which is not incident to any other tile. Thus, in any matching, there will be at least one edge incident to each tile. Then, note that the only tiles that have a pair parallel edges that are both on the boundary are $G_1$ and $G_m$. Since a matching of $G$ uses $m+1$ tiles, exactly one of these is twistable in a minimal or maximal matching. The specific cases follow from the definition of the minimal/maximal matching and direct/inverse zig-zags. 
\end{proof}

Since zig-zag snake graphs are well understood, it is to our advantage to think about general snake graphs as a composition of zig-zag shapes.  
Let $\mathcal{G} = (G_1,\ldots,G_n)$. We define a set of (overlapping) connected subsets of vertices, $S_1 = (G_1 = G_{i_0},\ldots,G_{i_1}),S_2 = (G_{i_1},\ldots,G_{i_2})$ and so on until $S_\ell = (G_{i_{\ell-1}},\ldots,G_n = G_{i_\ell})$ where the only values $k$ such that $G_{k-1},G_k,G_{k+1}$ form a straight line are $k = i_j$ for $0 < j < \ell$.

In the case of a cluster algebra from a surface this result follows from \cite{geiss2022schemes}; see Remark 11.1 and Proposition 10.14. See \cite{MSW-bases} for a similar result in the surface case concerning the $g$-vector of the cluster variable $x_\gamma$. 

\begin{proposition}\label{prop:gVectorMinMatch}
Fix an orbifold $\mathcal{O}$ with triangulation $T$.
\begin{enumerate}
    \item Let $\gamma$ be an arc on $\mathcal{O}$ and let $\mathcal{G}_{\gamma,T}$ be the snake graph from $\gamma$. Then, \[
\frac{\text{wt}(P_-)}{\text{cross}(\gamma,T)} = x^{\mathbf{g}_\Lambda(M(w_\gamma))}
\]
where $P_-$ is the minimal matching of $\mathcal{G}_{\gamma,T}$.
    \item Let $\xi$ be a closed curve on $\mathcal{O}$ and let $\widetilde{\mathcal{G}}_{\xi,T}$ be the band graph from $\xi$. Then, \[
    \frac{\text{wt}(P_-)}{\text{cross}(\xi,T)} = x^{\mathbf{g}_{\Lambda}(M(\overline{w}_\xi))}
    \]
where $P_-$ is the minimal matching of $\widetilde{\mathcal{G}}_{\xi,T}$.
\end{enumerate}
\end{proposition}

\begin{proof}
(1) Suppose $\gamma$ is an arc on $\mathcal{O}$ and $\gamma$ has crossings with arcs $\tau_{i_1},\ldots,\tau_{i_m} \in T$, with the order determined by an orientation placed on $\gamma$. Let $\mathcal{G}_{\gamma,T}$ have tiles $G_1,\ldots,G_m$. If $m = 1$, the claim can be shown by direct inspection, so we assume $m > 1$. 

It follows from the construction of a snake graph that, for each $2 \leq j \leq m-1$, there are exactly two boundary edges with label $x_{\tau_{i_j}}$. If $G_{j-1},G_j,G_{j+1}$ share a vertex, so that they are in the same zig-zag piece, the minimal matching will use exactly one of these edges. If $G_{j-1},G_j,G_{j+1}$ do not share a vertex, so that they form a straight line, then the minimal matching will either use both of these edges or neither. 

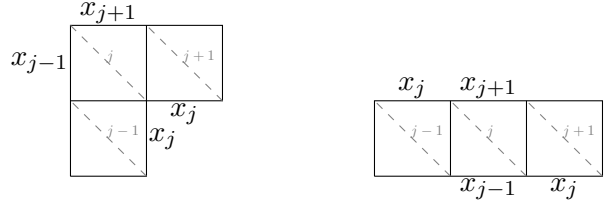
\begin{center}
\begin{tikzpicture}
\draw(0,0) to (1,0) to (1,1) to (2,1) to (2,2) to (1,2) to (0,2) to (0,1) to (0,0);
\draw (0,1) to (1,1) to (1,2);
\draw[gray,dashed] (1,0) to node[right,scale=0.5, xshift = -5pt, yshift = 5pt]{$j-1$} (0,1);
\draw[gray,dashed] (0,2) to node[right,scale=0.5, xshift = -5pt, yshift = 5pt]{$j$} (1,1);
\draw[gray,dashed] (2,1) to node[right,scale=0.5, xshift = -5pt, yshift = 5pt]{$j+1$} (1,2);
\node[right, xshift = -3] at (1,0.5){$x_{j}$};
\node[below, yshift = 3] at (1.5,1){$x_j$};
\node[left, xshift = 3] at (0,1.5){$x_{j-1}$};
\node[above, yshift = -3] at (0.5,2){$x_{j+1}$};
\draw[xshift = 5\R] (0,0) to (1,0) to (2,0) to (3,0) to (3,1) to (0,1) to (0,0);
\draw[xshift = 5\R] (1,0) to (1,1);
\draw[xshift = 5\R] (2,0) to (2,1);
\draw[gray,dashed,xshift = 5\R] (1,0) to node[right,scale=0.5, xshift = -5pt, yshift = 5pt]{$j-1$} (0,1);
\draw[gray,dashed,xshift = 5\R] (2,0) to node[right,scale=0.5, xshift = -5pt, yshift = 5pt]{$j$} (1,1);
\draw[gray,dashed,xshift = 5\R] (3,0) to node[right,scale=0.5, xshift = -5pt, yshift = 5pt]{$j+1$} (2,1);
\node[above, yshift = -3,xshift = 5\R ] at (0.5,1){$x_j$};
\node[below, yshift = 3, xshift = 5\R] at (2.5,0){$x_j$};
\node[below, xshift = 5\R, yshift = 3] at (1.5,0){$x_{j-1}$};
\node[above, xshift = 5\R, yshift = -3] at (1.5,1){$x_{j+1}$};
\end{tikzpicture}
\end{center}

There is only one boundary edge labeled $x_{i_1} = x_{\tau_{i_1}}$ which lies on $G_2$. Since we have established the convention that $P_-$ includes the south edge on $G_1$, we see that the edge labeled $x_{i_1}$ will be in $P_-$ only if $G_2$ is glued to the north side of $G_1$. There is similarly exactly one boundary edge labeled $x_{i_m} = x_{\tau_{i_m}}$ which lies on $G_{m-1}$.

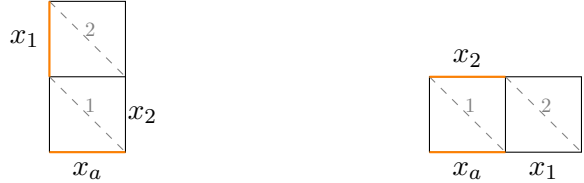
\begin{center}
\begin{tikzpicture}
\draw(0,0) to (1,0) to (1,1) to (1,2) to (0,2) to (0,1) to (0,0);
\draw (0,1) to (1,1);
\draw[gray,dashed] (0,1) to node[right, scale=0.7, xshift = -5pt, yshift = 5pt] {$1$} (1,0);
\draw[gray,dashed] (0,2) to node[right, scale=0.7, xshift = -5pt, yshift = 5pt] {$2$} (1,1);
\node[right, xshift = -3] at (1,0.5){$x_2$};
\node[left] at (0,1.5){$x_1$};
\node[below] at (0.5,0){$x_a$};
\draw[thick,orange] (0,0) to (1,0);
\draw[thick,orange] (0,1) to (0,2);
\draw (5,0) to (7,0) to (7,1) to (5,1) to (5,0);
\draw (6,0) to (6,1);
\node[below] at (5.5,0){$x_a$};
\node[below] at (6.5,0){$x_1$};
\node[above] at (5.5,1){$x_2$};
\draw[gray,dashed] (5,1) to node[right, scale=0.7, xshift = -5pt, yshift = 5pt] {$1$} (6,0);
\draw[gray,dashed] (6,1) to node[right, scale=0.7, xshift = -5pt, yshift = 5pt] {$2$} (7,0);
\draw[thick,orange] (5,0) to (6,0);
\draw[thick,orange] (5,1) to (6,1);
\end{tikzpicture}
\end{center}

We work with the convention that the orientation in $G_1$ matches the orientation on the orbifold. Therefore, the edge marked $x_a$ above will be labeled with the clockwise neighbor of $\tau_{i_1}$ in the first triangle, $\Delta_0$, $\gamma$ passes through, or if $\tau_{i_1}$ is a pending arc and $\gamma$ first passes through its interior, then this edge will also be labeled with $\tau_{i_1}$. In the former case, if this  clockwise neighbor is a boundary arc, the corresponding edge of the snake graph will not be labeled. There are similar cases for the north or east edge of $G_m$. 

If $w_\gamma = \alpha_1^{\epsilon_1}\cdots \alpha_{m-1}^{\epsilon_{m-1}}$,  and $c_1,\ldots,c_\ell$ are the deeps of $w_\gamma$, then the twistable tiles in $P_-$ are exactly the tiles of the form $G_{c_i}$. This follows from Theorem \ref{thm:CaScLatticeBij}. This means that none of the boundary edges labeled $x_{w(c_i)}$ are in $P_-$. Similarly, if $d_1,\ldots,d_{\ell'}$ are the peaks of $w_\gamma$, then all of the edges labeled $x_{w(d_i)}$ appear in $P_-$. If $d_i$ is a strict peak, there are two such edges and otherwise there is one such edge.  All other vertices $i$ of $w_\gamma$ correspond to vertices which sit in a zig-zag piece, so there will be one edge with label $x_{w(i)}$ in $P_-$. Finally, there will be factors $x_a,x_b$ coming from the first and last triangles or monogons $\gamma$ passes through, as discussed.  The exact form of $\frac{\text{wt}(P_-)}{\text{cross}(\gamma,T)}$ will depend on the signs $\epsilon_1$ and $\epsilon_{m-1}$ since only strict peaks will contribute to the numerator. For example, if $\epsilon_1 = \epsilon_{m-1} = +$, then $d_1 = 1$ is not a strict peak but $d_{\ell'}$ is a strict peak and the expression simplifies as \[
\frac{\text{wt}(P_-)}{\text{cross}(\gamma,T)} = \frac{x_a x_b x_{w(d_2)} \cdots x_{w(d_{\ell'})}}{x_{w(c_1)}\cdots x_{w(c_\ell)}}.
\]

If $\tau_a$ is a clockwise neighbor of $\tau_{i_1}$, then there is an arrow $\alpha: a \to i_1$ in $Q_T$ and $\alpha w_\gamma$ is a string. If $\gamma$ first crosses through a monogon enclosed by pending arc $\tau_{i_1}$, then there is a loop $\delta$ in $Q_T$ based at $i_1$ and $\delta w_\gamma$ is a string. There are similar cases for $\tau_b$.  With this observation, we conclude that $\frac{\text{wt}(P_-)}{\text{cross}(\gamma,T)} = x^{\mathbf{g}_\Lambda(M(w_\gamma))}$. Varying $\epsilon_1$ and $\epsilon_{m-1}$ only changes which peaks are strict and the same reasoning holds.

(2) Suppose $\overline{w}_{\xi} = \alpha_1^{\epsilon_1}\cdots \alpha_m^{\epsilon_m}$ and let $w = \alpha_1^{\epsilon_1}\cdots \alpha_{m-1}^{\epsilon_{m-1}}$. For convenience, choose a representative of $\overline{w}_{\xi}$ so that $\alpha_m$ is not a loop, implying that $s(w)$ and $t(w)$ are distinct vertices in $Q_0$. We know that $\tau_{s(w)}$ and $\tau_{t(w)}$ border the same triangle, which is both the first and the last triangle $\gamma(w)$ passes through. Assume without loss of generality that $\tau_{s(w)}$ immediately follows $\tau_{t(w)}$ in clockwise order in this triangle, and let $\tau_z$ be the third side of this triangle. Our choice of orientation between $\tau_{s(w)}$ and $\tau_{t(w)}$ implies that the arrow between $s(w)$ and $t(w)$ in $Q_T$ is orientated $s(w) \to t(w)$, so that $\epsilon_m = -$. Let $\mathcal{G}_{\gamma(w),T}$ be the snake graph for $\gamma(w)$. By the proof of part (1), we know that \[
\frac{\text{wt}(P^{\gamma(w)}_-)}{\text{cross}(\gamma(w),T)} = \frac{x_z x_{s(w)} x_{w(d_1)}\cdots x_{w(d_{\ell'})}}{ x_{w(c_1)}\cdots x_{w(c_\ell)}}
\]
where $P^{\gamma(w)}_-$ is the minimal matching of $\mathcal{G}_{\gamma(w),T}$, $c_1,\ldots,c_\ell$ are the deeps in $w$, and $d_1,\ldots,d_{\ell'}$ are the strict peaks. 

We glue $G_1$ and $G_m$ along the edges labeled $\tau_z$ in $\mathcal{G}_{\gamma(w),T}$ to form $\widetilde{\mathcal{G}}_{\xi,T}$. Our choice of orientation for $\tau_{s(w)},\tau_{t(w)},\tau_z$ means that the edge labeled $\tau_z$ on $G_1$ is in $P_-^{\gamma(w)}$, but the edge with the same label on $G_m$ will not be in the minimal matching. Therefore, this edge will not be in $P_-$, the minimal matching of $\widetilde{\mathcal{G}}_{\xi,T}$. All other edges will be the same, so we can conclude 
\[
\frac{\text{wt}(P_-)}{\text{cross}(\xi,T)} = \frac{ x_{s(w)} x_{w(d_1)}\cdots x_{w(d_{\ell'})}}{ x_{w(c_1)}\cdots x_{w(c_\ell)}}
\]
By analyzing cases for the signs $\epsilon_1,\epsilon_{m-1}$, we can see that the righthand side is equivalent to $x^{\mathbf{g}_\Lambda(M(\overline{w}_\xi))}$. For example, if $\epsilon_1 = \epsilon_{m-1} = +$, then the set of deeps of $w$ coincides with the set of strict deeps in $\overline{w}$ and similarly for strict peaks. If $\epsilon_1 = -$ and $\epislon_{m-1} = +$, then $1$ is a deep in $w$ and not in $\overline{w}_{\xi}$. Since there is a factor of $x_{s(w)}$ in both the numerator and denominator of the expression, it reduces to the expected result. 
\end{proof}

\subsection{Comparing Snake Graphs and CC Map}

 In this section we show that, given an unpunctured orbifold $\mathcal{O}$ with triangulation $T$, applying the Caldero-Chapoton map to any string or band for $(Q_T,I_T)$ is equivalent to using the snake graph expansion formula on the corresponding arc or closed curve on $\mathcal{O}$.

Given a quiver $Q$ with $\vert Q_0 \vert = n$, we define an $n\times n$ matrix $C_Q = (c_{i,j})$ where \[
c_{i,j} = \vert \{ \alpha \in Q_1 \vert \alpha: j \to i \} \vert - \vert \{\alpha \in Q_1 \vert \alpha: i \to j \} \vert.
\]

If there is a loop at vertex $i$, it will contribute 1 to each set, so that its contributions cancel. Thus, even when our quiver $Q_T$ has loops, its adjacency matrix is skew-symmetric.

\begin{lemma}\label{lem:GeneralMatchingWeight}
Let $P$ be a perfect matching of $\mathcal{G} = \mathcal{G}_{\gamma(w),T}$  with height monomial $\mathbf{h}(P)$.  Then, \[
\frac{\textit{wt}(P)}{\text{cross}(\gamma(w),T)} = \mathbf{x}^{\mathbf{g}_\Lambda(M(w)) + C_Q \mathbf{h}(P)}.
\]
The same is true for a good matching $P$ of $\widetilde{\mathcal{G}} = \widetilde{\mathcal{G}}_{\xi(\overline{w}),T}$ when we replace $\mathbf{g}_\Lambda(M(w))$ with $\mathbf{g}_\Lambda(M(\overline{w}))$ .
\end{lemma}

\begin{proof}
Given $\mathbf{h}(P) = (h_1,\ldots,h_n)$, we will induct on $h_1 + \cdots + h_n$. The case $h_1 = \cdots = h_n = 0$ occurs when $P$ is the minimal matching of $\mathcal{G}$, and the statement is true by Proposition \ref{prop:gVectorMinMatch}. Now assume we know the statement is true for $P'$ and consider $P$ which covers $P'$ in the lattice of perfect matchings of $\mathcal{G}$. This means that we reach $P$ from $P'$ by twisting one tile. Assume this tile is associated to $\tau_i \in T$. Then,  $\mathbf{h}(P) = \mathbf{h}(P') + \mathbf{e}_i$ where $\mathbf{e}_i$ is the $i$-th standard basis vector. 

Assume that the counterclockwise neighbors of $\tau_i$ are $\tau_b$ and $\tau_d$, and the clockwise neighbors are $\tau_a$ and $\tau_c$. By Lemma \ref{lem:twistparity} and the fact that tiles alternate in orientation,  $\text{wt}(P) = \text{wt}(P') \cdot \frac{x_b x_d}{x_a x_c}$. If some of these are boundary edges, their variables are set to 1. Recalling our convention for $C_Q$, we see that $\text{wt}(P) = \text{wt}(P') \cdot \mathbf{x}^{C_Q \mathbf{e}_i}$. Since by induction, $\text{wt}(P') = \mathbf{x}^{\mathbf{g}_\Lambda(M(w)) + C_Q \mathbf{h}(P')}$, our proof is complete. 

An identical proof shows the statement in the band case.

\end{proof}

Let $M$ be a representation of a bound quiver $(Q,I)$. For purposes of the definition, we consider $M$ as a module over the quotient of the completed path algebra $\Lambda = k\langle \langle Q \rangle \rangle /I$. The Caldero-Chapoton map of $M$, as defined in \cite{cerulli2015caldero}, is given by  \[
CC(M) = \mathbf{x}^{\mathbf{g}_\Lambda(M)} \sum_{\mathbf{e}} \chi(Gr_{\mathbf{e}}(M)) \mathbf{x}^{C_Q \mathbf{e}}.
\]

Since we want to compare this output with cluster variables in a generalized cluster algebra with principal coefficients, we will add coefficients to the definition of the Caldero-Chapoton map. See \cite{geiss2020generic} for a more in-depth discussion of the Caldero-Chapoton map with coefficients. Let $\tilde{\mathbf{x}} = (x_1,\ldots,x_n, y_1,\ldots,y_n)$ and  $\tilde{\mathbf{g}}_\Lambda(M) = (g_1,\ldots,g_n,0,\ldots,0)$ each be vectors of length $2n$. Given a quiver $Q$, with $\vert Q_0 \vert = n$, let $\widetilde{C_Q}$ be the $2n \times n$ matrix with top $n \times n$ submatrix given by $C_Q$ and bottom $n \times n$ submatrix given by the identity matrix $I_n$.  Then, we define the Caldero-Chapoton map with principal coefficients to be 

\[
CC_{prin}(M) = \tilde{\mathbf{x}}^{\tilde{\mathbf{g}}_\Lambda(M)} \sum_{\mathbf{e}} \chi(Gr_{\mathbf{e}}(M)) \tilde{\mathbf{x}}^{\widetilde{C_Q} \mathbf{e}}.
\]

\begin{theorem}\label{thm:CCMapAgreesSnakeGraph}
Consider an orbifold $\mathcal{O}$ with triangulation $T$. Let $(Q_T,I_T)$ be the gentle pair associated to $\mathcal{O},T$. 
\begin{enumerate}
\item Let $w$ be a string for $(Q_T,I_T)$ and let $\gamma(w)$ be the corresponding  arc in $\mathcal{O}$. Then,\begin{equation}\label{eq:SnakeArc}
\tilde{\mathbf{x}}^{\tilde{\mathbf{g}}_\Lambda(M(w))} \sum_{\mathbf{e}} \chi(Gr_{\mathbf{e}}(M(w))) \tilde{\mathbf{x}}^{\widetilde{C_{Q_T}} \mathbf{e}} = \frac{1}{\text{cross}(\gamma(w),T)} \sum_{P} \text{wt}(P)y(P)
\end{equation}
where the first sum is over $\mathbf{e} \in \mathbb{N}^{\vert Q_0 \vert}$ and the second sum is over perfect matchings of $\mathcal{G}_{\gamma(w),T}$. 
\item Let $\overline{w}$ be a band for $(Q_T,I_T)$ and let $\xi(\overline{w})$ be the corresponding  closed curve in $\mathcal{O}$. Then,\begin{equation}\label{eq:BandCurve}
\tilde{\mathbf{x}}^{\tilde{\mathbf{g}}_\Lambda(M(\overline{w}))} \sum_{\mathbf{e}} \chi(Gr_{\mathbf{e}}(M(\overline{w}))) \tilde{\mathbf{x}}^{\widetilde{C_{Q_T}} \mathbf{e}} = \frac{1}{\text{cross}(\xi(\overline{w}),T)} \sum_{P} \text{wt}(P)y(P)
\end{equation}
where the first sum is over $\mathbf{e} \in \mathbb{N}^{\vert Q_0 \vert}$ and the second sum is over good matchings of $\widetilde{\mathcal{G}}_{\xi(\overline{w}),T}$. 
\end{enumerate}
\end{theorem}

\begin{proof}
First, assume $w$ is a string with corresponding arc $\gamma(w)$. Consider the righthand side of Equation \ref{eq:SnakeArc}. By the definition of $y(P)$ and Lemma \ref{lem:GeneralMatchingWeight}, we can rewrite this as \[
\mathbf{x}^{\mathbf{g}_\Lambda(M(w))} \sum_P \mathbf{x}^{C_{Q_T}\mathbf{h}(P)} \mathbf{y}^{\mathbf{h}(P)} =
\tilde{\mathbf{x}}^{\tilde{\mathbf{g}}_\Lambda(M(w))} \sum_P \tilde{\mathbf{x}}^{\widetilde{C_{Q_T}}\mathbf{h}(P)}
\]

Then, if we let $n(\mathbf{e},\mathcal{G})$ denote the number of perfect matchings $P$ of a snake graph $\mathcal{G}$ with height vector $\mathbf{h}(P) = \mathbf{e}$, we can rewrite the sum and apply Corollary \ref{cor:CaScEulerChar} to reach our desired result,
\[
 =\tilde{\mathbf{x}}^{\tilde{\mathbf{g}}_\Lambda(M(w))} \sum_{\mathbf{e}} n(\mathbf{e},\mathcal{G}_{\gamma(w),T})\tilde{\mathbf{x}}^{\widetilde{C_{Q_T}}\mathbf{e}} =\tilde{\mathbf{x}}^{\tilde{\mathbf{g}}_\Lambda(M(w))} \sum_{\mathbf{e}} \chi(Gr_{\mathbf{e}}(M(w)))\tilde{\mathbf{x}}^{\widetilde{C_{Q_T}}\mathbf{e}} .
\]

We can reach our result in the band/closed curve case in the same way, using instead Corollary \ref{cor:EulerCharBandGraph} at the last step.

\end{proof}

Combining part 1 of Theorem \ref{thm:CCMapAgreesSnakeGraph} with Theorem 1.1 from \cite{banaian2020snake} yields the following.

\begin{corollary}\label{cor:CCOnOrdinaryIsClVar}
Consider an orbifold $\mathcal{O}$ with triangulation $T$. Let $(Q_T,I_T)$ be the gentle pair associated to $\mathcal{O},T$, and let $w$ be a string for $(Q_T,I_T)$ such that $\gamma(w)$ is an arc with no self-intersections. Then, $CC_{prin}(M(w))$ is the cluster variable $x_{\gamma(w)}$ in $\mathcal{A}(Q^T_{\text{gen}})$.
\end{corollary}

\begin{remark}
\begin{itemize}
    \item Corollary \ref{cor:CCOnOrdinaryIsClVar} was also given in the coefficient-free case by Labardini-Fragoso and Mou as Corollary 9.8 in \cite{labardini2023gentle}, using the theory of mutation of decorated representations. 
    \item Labardini-Fragoso and Mou show in \cite{labardini2023gentleI} that the Caldero-Chapoton map gives a bijection between $\tau$-rigid pairs of the algebra $KQ_T/I_T$ and cluster monomials from the generalized cluster algebra $\mathcal{A}(Q_{gen}^T)$.
\end{itemize}
\end{remark}

\subsection{Examples}

We end with examples of Theorem \ref{thm:CCMapAgreesSnakeGraph}. First, consider the following orbifold with triangulation $T$ and arc $\gamma$. 

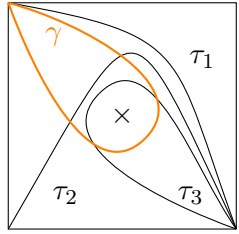
\begin{center}
\begin{tikzpicture}[scale=1.5]
\node[] (E) at (1,1) {$\times$};
 \draw[](2,0) -- (2,2) -- (0,2) -- (0,0) -- (2,0);
 \draw[](0,2) to[out = -10, in = 110, looseness = 1.5] (2,0);
 \node[right] at (1.5,1.5) {$\tau_1$};
 \draw[] (0,0) to [out = 60, in = 115, looseness=3] (2,0);
 \node[right] at (0.3,0.3){$\tau_2$};
 \draw[] (2,0) to [out=125, in = 45] (0.8,1.2);
 \draw[] (2,0) to [out = 160, in = 225](0.8,1.2);
 \node[] at (1.6,0.3){$\tau_3$};
 \draw[thick, orange] (0,2) to [out = -15, in = 45] (1.2,0.8);
 \draw[thick,orange] (0,2) to [out = 290,  in = 225](1.2,0.8);
 \node[orange] at (0.4,1.7){$\gamma$};
 \end{tikzpicture} 
 \end{center}
 
 The cluster variable associated to $\gamma$ is \[
  x_\gamma = \frac{x_3^3 + 2y_3x_1x_3^2 + y_3^2x_1^2x_3 + y_2y_3x_1x_2x_3^2 + y_2y_3^2x_1^2x_2x_3 + y_2^2y_3^2x_1^2x_2^2x_3}{x_2^2x_3^2}.
 \]
 
 The quiver $Q = Q_T$ is given by 
 
 
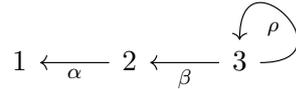
\begin{center}
 \begin{tikzcd}
 1 & 2 \arrow[l,"\alpha"] & 3 \arrow[l,"\beta"] \arrow[out=0,in=90,loop, "\rho"]\\
 \end{tikzcd}
 \end{center}
 
 and $I_T = \langle \rho^2 \rangle$. We see that $w_\gamma = \beta \rho \beta^{-1}$ or equivalently $\beta \rho^{-1}\beta^{-1}$. Choosing an orientation to match the latter string, the snake graph $\mathcal{G}_{\gamma,T}$ is below. 
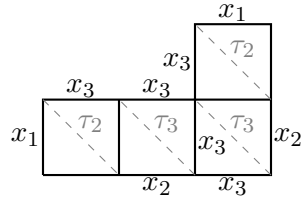
\[
 \begin{tikzpicture}
 \draw[thick] (0,0) to (1,0) to node[below, yshift = 3pt]{$x_2$} (2,0) to node[below, yshift = 3pt]{$x_3$} (3,0) to node[right, xshift = -3pt]{$x_2$} (3,1) to (3,2) to node[above, yshift = -3pt]{$x_1$} (2,2) to node[left, xshift = 3pt]{$x_3$} (2,1) to node[above, yshift = -3pt]{$x_3$}(1,1) to node[above, yshift = -3pt]{$x_3$} (0,1) to node[left, xshift = 3pt]{$x_1$} (0,0);
\draw[thick] (1,0) -- (1,1);
\draw[thick] (2,0) to node[right, yshift = -3pt, xshift = -3pt]{$x_3$} (2,1) to (3,1);
\draw[gray,dashed] (0,1) to node[midway,right, xshift = -5pt, yshift = 5pt]{$\tau_2$} (1,0);
\draw[gray,dashed] (2,0) to node[midway,right,xshift = -5pt, yshift = 5pt]{$\tau_3$} (1,1);
\draw[gray,dashed] (3,0) to node[midway,right,xshift = -5pt, yshift = 5pt]{$\tau_3$} (2,1);
\draw[gray,dashed] (3,1) to node[midway,right,xshift = -5pt, yshift = 5pt]{$\tau_2$} (2,2);
 \end{tikzpicture}
 \]
 
 The lattice of perfect matchings and the corresponding order ideals of the fence poset were given in Example \ref{ex:SnakeAndPoset}. The module $M = M(w_\gamma)$ is given by 
 
 \begin{center}
\begin{tikzcd}
0 & k^2 \arrow{l}{} & k^2 \arrow{l}{ \begin{pmatrix} 1 \amsamp 0 \\ 0 \amsamp 1 \end{pmatrix}}\arrow[loop right]{r}{\begin{pmatrix} 0 \amsamp 0 \\ 1 \amsamp 0 \end{pmatrix}} 
\end{tikzcd}
\end{center}

The $g$-vector of $M$ is $\mathbf{g}_\Lambda(M) =(0,-2,1)$. We check\[
\frac{\text{wt}(P_-)}{\text{cross}(\gamma,T)} = \frac{x_3^3}{x_2^2x_3^2} = \frac{x_3}{x_2^2} = \mathbf{x}^{\mathbf{g}_\Lambda(M). }\] 

Finally, we provide $\widetilde{C_{Q}}$, \[
\begin{pmatrix} 0 & 1 & 0 \\
-1 & 0 & 1\\
0 & -1 & 0 \\
1 & 0 & 0 \\
0 & 1 & 0\\
0 & 0 & 1\\
\end{pmatrix}
\]

The following table summarizes the calculation on the lefthand side of Equation \ref{eq:SnakeArc}. Let $M = M(w_\gamma)$. We omit vectors $\mathbf{e}$ such that $\text{Gr}_{\mathbf{e}}(M) = \emptyset$ since these do not contribute to the sum.

\begin{center}
\renewcommand{\arraystretch}{2}
\begin{tabular}{|c|c|c|}\hline
    $\mathbf{e}$ & $\chi(\text{Gr}_{\mathbf{e}}(M))$ &  $\widetilde{\mathbf{x}}^{C_{Q}\mathbf{e}}$ \\ \hline\hline
    (0,0,0) & 1 & 1\\\hline
    (0,1,0) & 2 & $\frac{x_1}{x_3}y_2$\\\hline
    (0,2,0) & 1 & $\frac{x_1^2}{x_3^2}y_2^2$\\\hline
    (0,1,1) & 1 & $\frac{x_1x_2}{x_3}y_2y_3$\\\hline
    (0,2,1) & 1 & $\frac{x_1^2x_2}{x_3^2}y_2^2y_3$\\\hline
    (0,2,2) & 1 & $\frac{x_1^2x_2^2}{x_3^2}y_2^2y_3^2$\\\hline
\end{tabular}
 \end{center}
 
 Comparing the first two columns with the lattice in Example \ref{ex:SnakeAndPoset} illustrates \canakci and Schroll's result (Theorem \ref{thm:CaScLatticeBij}). For example, $\chi(\text{Gr}_{(0,1,0)}(M)) = 2$ and there are two matchings of $\mathcal{G}_{\gamma,T}$ with height vector $(0,1,0)$. 
 
There are no bands in the gentle algebra given in the first example, and likewise there are no closed curves in the orbifold which are neither contractible nor enclose a single orbifold point. We look at a second triangulated orbifold which will admit interesting closed curves, one of which is depicted below and labeled $\xi$. 

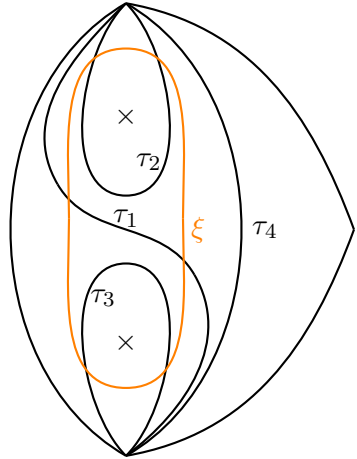
\begin{center}
\begin{tikzpicture}[scale=3]
\draw[thick] (0,0.5) node {$\mathbf{\times}$};
\draw[thick] (0,1.5) node {$\mathbf{\times}$};
\draw[in=0,out=45,looseness =0.9, thick] (0,0) to  (0,0.85);
\draw[in=180,out=135,looseness=0.9, thick] (0,0) to (0,0.85);
\draw[in=0,out=180+135,looseness =0.9, thick] (0,2) to (0,1.15);
\draw[in=180,out=180+45,looseness=0.9, thick](0,2) to (0,1.15);
\draw[out=30,in=-30, thick] (0,0) to node[right] {$\tau_4$} (0,2);
\draw[out = 10, in = 250, thick] (0,0) to node[right]{} (1,1);
\draw[out = 110, in = -10, thick] (1,1) to node[right]{} (0,2);
\draw[out=150,in=-150, thick] (0,0) to node[left] {} (0,2);
\draw[out = 37, in = 300, thick] (0,0) to (0.3,0.8);
\draw[out = 120, in = 300, thick] (0.3,0.8) to  (-0.3, 1.2);
\draw[out = 120, in = 180+ 44, thick] (-0.3,1.2) to (0,2);
\draw[thick, orange, out = 0, in = 270] (0,0.3) to (0.25,1.05);
\draw[thick, orange, out = 90, in = 0] (0.25,1.05) to (0,1.8);
\draw[thick, orange, out = 180, in = 90] (0,1.8) to (-0.25,1.05);
\draw[thick, orange, out = 270, in = 180] (-0.25,1.05) to (0,0.3);
\node[] at (0,1.05){$\tau_1$};
\node[] at (0.1,1.3){$\tau_2$};
\node[] at (-0.1,0.7){$\tau_3$};
\node[orange, right, xshift = -5pt] at (0.3,1){$\xi$};
\end{tikzpicture}
\end{center}

The quiver $Q = Q_T$ for this triangulation is given by 

\begin{center}
\begin{tikzcd}[arrow style=tikz,>=stealth,row sep=1.5em]
3  \arrow[out=90,in=180, loop, "\epsilon"] \arrow[r, "\delta"] & 1 \arrow[dr, "\beta"] && 2  \arrow[out=90,in=0, loop, "\mu"] \arrow[ll ,"\alpha"]\\
& & 4 \arrow[ur, "\gamma"] & \\
\end{tikzcd} 
\end{center}

and $I_T = \langle \alpha \beta, \beta \gamma, \gamma \alpha, \epsilon^2, \mu^2 \rangle$. The band associated to $\xi$ is $\overline{w}_\xi = \alpha^{-1} \mu \alpha \delta^{-1}\epsilon \delta $. For $\lambda \in k^\times$, the band module $M(\overline{w}_\xi, \lambda, 1)$ looks as follows

\begin{center}
\begin{tikzcd}[arrow style=tikz,>=stealth,row sep=1.5em]
k^2  \arrow[loop left]{l}{\begin{pmatrix} 0 \amsamp 0 \\ 1 \amsamp 0 \end{pmatrix}} \arrow[right]{r}{\begin{pmatrix} 0 \amsamp \lambda \\ 1 \amsamp 0 \end{pmatrix}} & k^2 \arrow[dr, "0"] && k^2 \arrow[loop right]{r}{\begin{pmatrix} 0 \amsamp 0 \\ 1 \amsamp 0 \end{pmatrix}} \arrow{ll}[above]{\begin{pmatrix} 1 \amsamp 0\\ 0 \amsamp 1 \end{pmatrix}}\\
& & 0 \arrow[ur, "0"] & \\
\end{tikzcd} 
\end{center}

The band graph $\widetilde{G}_{\xi,T}$ is given below, where we identify vertices $a$ and $a'$ and $b$ and $b'$, thereby gluing the graph along the bold purple edges. 

\begin{center}
\begin{tikzpicture}[scale = 1.2]
\draw(0,0) -- (3,0) -- (3,4) -- (2,4) -- (2,1) -- (0,1) -- (0,0);
\draw (1,0) -- (1,1);
\draw (2,0) -- (2,1) -- (3,1);
\draw (2,2) -- (3,2);
\draw (2,3) -- (3,3);
\draw[gray,dashed] (1,0) to node[right]{$\tau_1$} (0,1);
\draw[gray,dashed] (2,0) to node[right]{$\tau_2$} (1,1);
\draw[gray,dashed] (3,0) to node[right]{$\tau_2$} (2,1);
\draw[gray,dashed] (3,1) to node[right]{$\tau_1$} (2,2);
\draw[gray,dashed] (3,2) to node[right]{$\tau_3$} (2,3);
\draw[gray,dashed] (3,3) to node[right]{$\tau_3$} (2,4);
\node[below, yshift = 3pt] at (0.5,0){$x_3$};
\node[below, yshift = 3pt] at (1.5,0){$x_1$};
\node[below, yshift = 3pt] at (2.5,0){$x_2$};
\node[right, xshift = -3pt] at (3,0.5){$x_1$};
\node[right, xshift = -3pt] at (3,1.5){$x_3$};
\node[right, xshift = -3pt] at (3,2.5){$x_3$};
\node[left, xshift = 3pt] at (2,1.5){$x_2$};
\node[left, xshift = 3pt] at (2,2.5){$x_1$};
\node[left, xshift = 3pt] at (2,3.5){$x_3$};
\node[above, yshift = -3pt] at (0.5,1){$x_2$};
\node[above, yshift = -3pt] at (1.5,1){$x_2$};
\node[right,xshift = -3pt] at (1,0.5){$x_4$};
\node[right,xshift = -3pt] at (2,0.5){$x_2$};
\node[above,yshift = -3pt] at (2.5,1){$x_4$};
\node[above,yshift = -3pt] at (2.5,3){$x_3$};
\node[above,yshift = -3pt] at (2.5,4){$x_1$};
\draw[ultra thick, purple] (0,0) -- (0,1);
\draw[ultra thick, purple] (3,3) -- (3,4);
\node[left, xshift = 2pt] at (0,0){$a$};
\node[left,xshift = 2pt] at (0,1){$b$};
\node[right, xshift = -2pt] at (3,3){$a'$};
\node[right,xshift = -2pt] at (3,4){$b'$};
\end{tikzpicture}
\end{center}

The $g$-vector of $M(\overline{w}_{\xi}, \lambda,1)$ is $(-2,1,1,0)$, and one can again check that this matches $\frac{wt(P_-)}{\text{cross}(\xi,T)}$. We provide $\widetilde{C_Q}$ \[
\begin{pmatrix}
0 & 1 & 1 & -1 \\
-1 & 0 & 0 & 1\\
-1 & 0 & 0 & 0 \\
1 & - 1 & 0 & 0 \\
1 & 0 & 0 & 0\\
0 & 1 & 0 & 0\\
0 & 0 & 1 & 0\\
0 & 0 & 0 & 1\\
\end{pmatrix}.
\]

Finally we again provide a table to summarize the calculation of the Caldero-Chapoton map applied to $M = M(\overline{w}_{\xi}, \lambda,1)$. 

\begin{center}
\renewcommand{\arraystretch}{1.5}
\begin{tabular}{|c|c|c||c|c|c|}\hline
    $\mathbf{e}$ & $\chi(\text{Gr}_{\mathbf{e}}(M))$ &  $\widetilde{\mathbf{x}}^{C_{Q}\mathbf{e}}$&   $\mathbf{e}$ & $\chi(\text{Gr}_{\mathbf{e}}(M))$ &  $\widetilde{\mathbf{x}}^{C_{Q}\mathbf{e}}$ \\\hline\hline
    (0,0,0,0) & 1 & 1& (2,2,0,0) & 1 & $\frac{x_1^2}{x_2^2x_3^2}y_1^2y_2^2$\\\hline
    (1,0,0,0) & 2 & $\frac{x_4}{x_2x_3}y_1$&(2,1,1,0) & 1 & $\frac{x_1^2x_4}{x_2^2x_3^2}y_1^2y_2y_3$\\\hline
    (2,0,0,0) & 1 & $\frac{x_4^2}{x_2x_3}y_1^2$& (2,0,2,0) & 1 & $\frac{x_1^2x_4^2}{x_2^2x_3^2}y_1^2y_3^2$ \\\hline
    (1,1,0,0) & 1 & $\frac{x_1}{x_2x_3}y_1y_2$&(2,2,1,0) & 1 & $\frac{x_1^3}{x_2^2x_3^2}y_1^2y_2^2y_3$\\\hline
    (1,0,1,0) & 1 & $\frac{x_1x_4}{x_2x_3}y_1y_3$& (2,1,2,0) & 1 & $\frac{x_1^3x_4}{x_2^2x_3^2}y_1^2y_2y_3^2$\\\hline
    (2,1,0,0) & 1 & $\frac{x_1x_4}{x_2^2x_3^2}y_1^1y_2$& (2,2,2,0) & 1 & $\frac{x_1^4}{x_2^2x_3^2}y_1^2y_2^2y_3^2$\\\hline
    (2,0,1,0) & 1 & $\frac{x_1x_4^2}{x_2^2x_3^2}y_1^2y_3$&&&\\\hline
\end{tabular}
 \end{center}


\begin{thebibliography}{50}

\bibitem{amiot2009cluster}
Amiot, Claire.  ``Cluster categories for algebras of global dimension 2 and quivers with potential.''  \textit{Annales de l'institut Fourier. } Vol. 59. No. 6. 2009.

\bibitem{assem2006elements} Assem, Ibrahim, Daniel Simson, and Andrzej Skowronski.  \textit{Elements of the Representation Theory of Associative Algebras: Volume 1: Techniques of Representation Theory.} Cambridge University Press, 2006.
 
\bibitem{assem2010gentle} Assem, Ibrahim, Thomas Br{\"u}stle, Gabrielle Charbonneau-Jodoin, and Pierre-Guy Plamondon ``Gentle algebras arising from surface triangulations.'' \textit{Algebra \& Number Theory} 4.2 (2010): 201-229.

\bibitem{banaian2020snake} Banaian, Esther and Elizabeth Kelley. ``Snake graphs from triangulated orbifolds.'' \textit{Symmetry,  Integrability and Geometry: Methods and Applications} 16 (2020): 138.

\bibitem{baur2021geometric} Baur, Karin and Raquel Coelho Sim{\~o}es. ``A geometric model for the module category of a gentle algebra.'' \textit{International Mathematics Research Notices} 2021.15 (2021): 11357-11392.

\bibitem{brustle2011cluster}
Br{\"u}stle, Thomas and Jie Zhang. ``On the cluster category of a marked surface without punctures.'' \textit{Algebra \& Number Theory} 5.4 (2011): 529-566.


\bibitem{buan2006tilting} Buan, Aslak Bakke, Bethany Rose Marsh, Markus Reineke, Idun Reiten, and Gordana Todorov.  ``Tilting theory and cluster combinatorics.''  \textit{Advances in mathematics} 204.2 (2006): 572-618.



\bibitem{butler1987auslander}
Butler, Michael CR and Claus Michael Ringel. ``Auslander-reiten sequences with few middle terms and applications to string algebras.'' \textit{Communications in Algebra} 15.1-2 (1987): 145-179.


\bibitem{CCMap}
Caldero, Philippe and Frédéric Chapoton.  ``Cluster algebras as Hall algebras of quiver representations.'' \textit{Commentarii Mathematici Helvetici }81.3 (2006): 595-616.


\bibitem{ccanakcci2020snake}
{\c{C}}anak{\c{c}}i, {\.I}lke and Ralf Schiffler.  ``Snake graphs and continued fractions.''  \textit{European Journal of Combinatorics} 86 (2020): 103081.

\bibitem{ccanakcci2021lattice}
{\c{C}}anak{\c{c}}i, {\.I}lke and Sibylle Schroll.  ``Lattice bijections for string modules, snake graphs and the weak Bruhat order.'' \textit{Advances in Applied Mathematics} 126 (2021): 102094.


\bibitem{cerulli2015caldero}
Cerulli Irelli, Giovanni, Daniel Labardini-Fragoso, and Jan Schr{\"o}er, ``Caldero-Chapoton algebras.'' \textit{Transactions of the American Mathematical Society} 367.4 (2015): 2787-2822.


\bibitem{chekhov2014teichmuller} Chekhov, Leonid and Michael Shapiro, ``Teichm{\"u}ller spaces of Riemann surfaces with orbifold points of arbitrary order and cluster variables.'' \textit{International Mathematics Research Notices} 2014.10 (2014) 2746--2772.


\bibitem{FST-orbTriang} Felikson, Anna, Michael Shapiro, and Tumarkin, Pavel, ``Cluster Algebras and Triangulated Orbifolds.'' \textit{Advances in Mathematics} 231.5 (2012) 2953--3002.

	
\bibitem{FST-triangulatedSurfaces} Fomin, Sergey,  Michael Shapiro, and Thurston, Dylan, ``Cluster algebras and triangulated surfaces, part I: Cluster complexes.'' \textit{Acta Mathematica} 201.1 (2008) 83--146.

	
\bibitem{Fomin-Zelevinsky-I} Fomin, Sergey and Andrei Zelevinsky, ``Cluster Algebras  I: Foundations.'' \textit{Journal of the American Mathematical Society} 15.2 (2002) 497--529.

	
\bibitem{fomin2007cluster} Fomin, Sergey and Andrei Zelevinsky,  ``Cluster Algebras IV: coefficients.'' \textit{Compositio Mathematica} 143.1 (2007) 112--164. 




\bibitem{geiss2020generic} Gei{\ss}, Christof, Daniel Labardini-Fragoso, and Jan Schr{\"o}er,  ``Generic Caldero-Chapoton functions with coefficients and applications to surface cluster algebras.'' \textit{arXiv preprint arXiv:2007.05483} (2020).


\bibitem{geiss2022schemes} Gei{\ss}, Christof, Daniel Labardini-Fragoso, and Jan Schr{\"o}er,  ``Schemes of modules over gentle algebras and laminations of surfaces.'' \textit{Selecta Mathematica} 28.1 (2022).

\bibitem{haupt2012euler} Haupt, Nicolas,  ``Euler characteristics of quiver Grassmannians and Ringel-Hall algebras of string algebras.'' \textit{Algebras and representation theory} 15.4 (2012) 755--793. 

\bibitem{labardini2009quivers} Labardini-Fragoso, Daniel, ``Quivers with potentials associated to triangulated surfaces.'' \textit{Proceedings of the London Mathematical Society} 98.3 (2009) 797--839. 

\bibitem{labardini2023gentleI}  Labardini-Fragoso, Daniel and Lang Mou,  ``Gentle algebras arising from surfaces with orbifold points of order 3, Part I: scattering diagrams.'' \textit{Algebras and Representation Theory} (2023) 1--44. 



\bibitem{labardini2023gentle} Labardini-Fragoso, Daniel and Lang Mou,  ``Gentle algebras arising from surfaces with orbifold points of order 3, Part II:  Locally free Caldero—Chapoton functions.'' \textit{arXiv preprint arXiv:2309.16061} (2023).


\bibitem{labardini2019family} Labardini-Fragoso, Daniel and Diego Velasco,  ``On a family of Caldero--Chapoton algebras that have the Laurent phenomenon.'' \textit{Journal of Algebra}, 520 (2019) 90--135.

\bibitem{Musiker-Schiffler}  Musiker, Gregg and Ralf Schiffler,  ``Cluster expansion formulas and perfect matchings.'' \textit{Journal of Algebraic Combinatorics},32 (2010) 187--209.

\bibitem{MSW-bases} Musiker, Gregg,  Ralf Schiffler, and Lauren Williams, ``Bases for cluster algebras from surfaces.'' \textit{Compositio Mathematica},  149.2 (2013).

\bibitem{MSW} Musiker, Gregg,  Ralf Schiffler, and Lauren Williams,  ``Positivity for Cluster Algebras from Surfaces.'' \textit{Advances in Mathematics}, 227.6 (2011) 2241--2308.
	
\bibitem{Nakanishi} Nakanishi, Tomoki,  ``Structure of Seeds in Generalized Cluster Algebras.'' \textit{Pacific Journal of Mathematics}, 277.1 (2015) 201--218.
		

\bibitem{opper2018geometric} Opper, Sebastian,  Pierre-Guy Plamondon,  and Sibylle Schroll,  ``A geometric model for the derived category of gentle algebras.'' \textit{arXiv preprint arXiv:1801.09659} (2018).



\bibitem{palu1707non} Palu, Yann, Vincent Pilaud,  and Pierre-Guy Plamondon,  ``Non-kissing complexes and $\tau$-tilting for gentle algebras.'' \textit{Memoirs of the American Mathematical Society} 274.1343 (2021).

	
\bibitem{schiffler2014quiver}  Schiffler, Ralf,  \textit{Quiver Representations}. Springer, 2014. 

\bibitem{williams2014cluster}  Williams,  Lauren,  ``Cluster algebras: an introduction.'' \textit{Bulletin of the American Mathematical Society}, 51.1 (2014) 1--26.
\end{thebibliography}
\end{document}